\documentclass[11pt]{amsart}
\usepackage{amsmath,amsfonts,amssymb,amscd,mathtools,bm}
\usepackage{amsrefs}
\usepackage[colorlinks=true, linkcolor=blue,
  citecolor=blue, urlcolor=blue]{hyperref}  
\usepackage{tikz}
\usepackage{tikz-cd} 
\usepackage{enumitem}
\usepackage{multirow}
\usepackage{caption}
\usepackage{subcaption}

\def\newthm#1#2{\newtheorem{#1}[dummy]{#2}%
  \expandafter\def\csname#2\endcsname##1{\hyperref[#1:##1]{#2~\ref*{#1:##1}}}}
\newthm{thm}{Theorem}
\newthm{lemma}{Lemma}
\newthm{prop}{Proposition}
\newthm{cor}{Corollary}
\newthm{conj}{Conjecture}
\newthm{cons}{Consequence}
\newthm{fact}{Fact}
\newthm{obs}{Observation}
\newthm{guess}{Guess}
\newthm{algo}{Algorithm}
\theoremstyle{definition}
\newthm{defn}{Definition}
\newthm{remark}{Remark}
\newthm{example}{Example}
\newthm{question}{Question}
\newthm{notation}{Notation}

\newcommand{\Section}[1]{\hyperref[sec:#1]{Section~\ref*{sec:#1}}}
\newcommand{\Table}[1]{\hyperref[tab:#1]{Table~\ref*{tab:#1}}}
\newcommand{\eqn}[1]{\hyperref[eqn:#1]{(\ref*{eqn:#1})}}
\newcommand{\Figure}[1]{\hyperref[fig:#1]{Figure~\ref*{fig:#1}}}

\makeatletter
\def\namedlabel#1#2{\begingroup
    #2%
    \def\@currentlabel{#2}%
    \phantomsection\label{#1}\endgroup
}
\makeatother

\DeclareMathOperator{\SL}{SL}

\DeclareMathOperator{\Gr}{Gr}
\DeclareMathOperator{\Fl}{Fl}

\DeclareMathOperator{\pt}{pt}

\DeclareMathOperator{\ch}{\ch}

\usepackage{bm}

\renewcommand{\P}{{\mathbb P}}
\newcommand{\C}{{\mathbb C}}

\newcommand{\Z}{{\mathbb Z}}
\newcommand{\N}{{\mathbb N}}

\newcommand{\cF}{{\mathcal F}}

\newcommand{\cO}{{\mathcal O}}

\DeclareMathOperator{\ev}{ev}

\newcommand{\wb}{\overline}

\newcommand{\ignore}[1]{}

\newcommand{\Mb}{\wb{\mathcal M}}

\newcommand{\M}{\overline{M}}

\usepackage[margin=1in]{geometry}

\begin{document}

\title
[Quantum \(K\)-theory of Incidence Varieties]
{Quantum \(K\)-theory of Incidence Varieties}

\date{\today}

\author{Weihong Xu}
\address{Department of Mathematics, 
Virginia Tech, 460 McBryde Hall, 225 Stanger Street, Blacksburg, VA 24061-1026, USA}
\email{weihong@vt.edu}


\begin{abstract}
    We prove a conjecture of Buch and Mihalcea in the case of the incidence variety \(X=\Fl(1,n-1;n)\) and determine the structure of its (\(T\)-equivariant) quantum \(K\)-theory ring. Our results are an interplay between geometry and combinatorics. The geometric side concerns Gromov--Witten varieties of \(3\)-pointed genus \(0\) stable maps to \(X\) with markings sent to Schubert varieties, while on the combinatorial side are formulas for the (equivariant) quantum \(K\)-theory ring of \(X\). We prove that the Gromov--Witten variety is rationally connected when one of the defining Schubert varieties is a divisor and another is a point. This implies that the (equivariant) \(K\)-theoretic Gromov--Witten invariants defined by two Schubert classes and a Schubert divisor class can be computed in the ordinary (equivariant) \(K\)-theory ring of \(X\). We derive a positive Chevalley formula for the equivariant quantum \(K\)-theory ring of \(X\) and a positive closed formula for Littlewood--Richardson coefficients in the non-equivariant quantum \(K\)-theory ring of \(X\). The Littlewood--Richardson rule in turn implies that non-empty Gromov--Witten varieties given by Schubert varieties in general position have arithmetic genus \(0\).
\end{abstract}

\maketitle

\section{Introduction}

Let \(X\) be a homogeneous space \(G/P\) defined by a complex semisimple linear algebraic group \(G\) and a parabolic subgroup \(P\). Let \(d\in H_2(X,\Z)\) be an effective degree and \(M_d\coloneqq\M_{0,3}(X,d)\) be the Kontsevich moduli space parametrizing \(3\)-pointed, genus \(0\), degree \(d\) stable maps to \(X\). Let \(X_u\) and \(X^v\) be opposite Schubert varieties in \(X\). Let \[M_d(X_u,X^v)\coloneqq \ev_1^{-1}(X_u)\cap \ev_2^{-1}(X^v)\subseteq M_d\] be the Gromov--Witten variety of stable maps with one marking sent to \(X_u\) and another to \(X^v\), and \[\Gamma_d(X_u,X^v)\coloneqq \ev_3(M_d(X_u,X^v))\] be the closure of the union of degree \(d\) rational curves in \(X\) connecting \(X_u\) and \(X^v\). 

Buch, Chaput, Mihalcea, and Perrin \cite{BCMP} proved that when \(X\) is a cominuscule variety (certain homogeneous spaces with Picard rank one), \begin{equation}\label{eqn:ev3}
\ev_3: M_d(X_u,X^v)\to\Gamma_d(X_u,X^v)
\end{equation} is \emph{cohomologically trivial}. This means the pushforward of the structure sheaf of \(M_d(X_u,X^v)\) is the structure sheaf of \(\Gamma_d(X_u,X^v)\) and the higher direct images of the first sheaf are zero. As a consequence, the (\(T\)-equivariant) \(K\)-theoretic (\(3\)-pointed, genus \(0\)) Gromov--Witten invariants of \(X\) satisfy the identity 
\begin{equation}\label{eqn:comin}
    I_d^T([\cO_{X_u}],[\cO_{X^v}],\sigma)=\chi_X^T([\cO_{\Gamma_d(X_u,X^v)}]\cdot\sigma),
\end{equation} 
where \(\chi_X^T\) denotes the \(T\)-equivariant pushforward along the structure morphism of \(X\). Identity \eqn{comin} says that these invariants defined on \(M_d\) can be computed in the equivariant ordinary \(K\)-theory ring of \(X\). 
     
When \(X\) is not cominuscule, \eqn{comin} doesn't always hold (see \cite{gwgrass}*{remark 2} for a counterexample).
However, Buch and Mihalcea conjecture the following \cite{private}.
\begin{conj}\label{conj:conj} 
    When \(X\) is of type \(A\) and \(X^v\) is a divisor \(D\), the map \eqn{ev3} is cohomologically trivial.
\end{conj}

In this paper, we consider the incidence variety \(X=\Fl(1,n-1;n)\), which is a special two-step flag variety of type \(A\) and one of the simplest homogeneous spaces that are not cominuscule.  

Our main geometric result is the following.
\begin{thm}[\(\doteq\) \Theorem{ratcon}]\label{thm:geom}
    The general fibre of \[\ev_3:M_d(X_u,D)\to\Gamma_d(X_u,D)\] is rationally connected, where \(D\) is an opposite Schubert divisor.
\end{thm} 
Using a result of Koll\'{a}r \cite{kollar1986higher} (see \cites{qkgrass,Buch2013}), \Theorem{geom} implies \Conjecture{conj} and the identity \begin{equation}\label{eqn:qclassical}I_d^T([\cO_{X_u}],[\cO_{D}],\sigma)=\chi_X^T([\cO_{\Gamma_d(X_u,D)}]\cdot\sigma).\end{equation} The right-hand side of \eqn{qclassical} is easily computable using Lenart and Postnikov's Chevalley formula for the equivariant ordinary \(K\)-theory ring of \(G/P\) \cite{LP}, because \(\Gamma_d(X_u,D)\) is either a Schubert variety or the intersection of a Schubert variety with \(D\) (see \Section{curvenbhd} for details).

Our main application is in studying the (\(T\)-equivariant) quantum \(K\)-theory ring of \(X\). Quantum \(K\)-theory was introduced by Givental and Lee \cites{givental2000wdvv,lee2004quantum} as a \(K\)-theoretic analogue of quantum cohomology. The latter has been widely studied in theoretical physics and various branches of mathematics following the pioneering works of Witten \cite{witten1990two} and Kontsevich \cite{kontsevich1995enumeration}. Rationality properties of Gromov--Witten varieties often have implications in quantum \(K\)-theory, see for example \cites{lee2004reconstruction, qkgrass, buch2013rational, chaput, Buch2013}.

Before stating our next main results, we shall introduce a bit more notations. (Opposite) Schubert varieties in \(X\) are indexed by \[W^P\coloneqq\{[i,j]: i\text{ and }j\text{ are unequal integers between }1\text{ and }n\}.\] Here \([i,j]\) corresponds to the permutation \(w\in S_n\) such that \(w(1)=i,\ w(n)=j\) and \(w(2)<\dots<w(n-1)\). In order to state our formulas in a simple form, we define \[\widetilde{W^P}\coloneqq\{[i,j]\in\Z\times\Z: i\not\equiv j\text{ mod }n\},\] and for \(w=[i,j]\in\widetilde{W^P}\) we define \[\cO^{w}\coloneqq q^{d(w)}[\cO_{X^{\overline w}}]\in {QK_T(X)}_q,\] where \(\overline{w}\coloneqq[\overline{i},\overline{j}]\in W^P\) is defined by \(\overline{i}\equiv i\) and \(\overline{j}\equiv j\) (mod \(n\)) being the unique such integers between \(1\) and \(n\), \[d(w)\coloneqq(\frac{i-\overline{i}}{n},\frac{\overline{j}-j}{n})\in H_2(X,\Z)=\Z\oplus\Z,\] and \[QK_T(X)_q\coloneqq QK_T(X)\otimes_{\Z[q]}\Z[q,q^{-1}]\] denotes the \(T\)-equivariant quantum \(K\)-theory ring of \(X\) with the deformation parameters \(q\) inverted. We let \(\cO^{[1]}\) and \(\cO^{[2]}\) be the classes of the structure sheaves of the two opposite Schubert divisors. 

Identity \eqn{qclassical} allows us to compute \(K\)-theoretic Gromov--Witten invariants and derive the following Chevalley formula for the equivariant quantum \(K\)-theory ring of \(X\). By \cite{qkchev}*{Section 5.3}, this formula uniquely determines the ring struture. Setting \(q=0\) recovers Lenart and Postnikov's Chevalley formula for equivariant ordinary \(K\)-theory \cite{LP} (see also \cite{knutson2017schubert} for a Littlewood--Richardson rule for the equivariant ordinary \(K\)-theory ring of two-step flag varieties). 
\begin{thm}[\(\doteq\) \Theorem{qkTchev}]\label{thm:chev}
    In \({QK_T(X)}\), for \([i,j]\in {W^P}\), the quantum multiplication by divisor classes is given by
    \begin{equation*}\cO^{[i,j]}\star\cO^{[1]}=\begin{cases}(1-[\C_{\varepsilon_i-\varepsilon_1}])\cO^{[i,j]}+[\C_{\varepsilon_i-\varepsilon_1}]\cO^{[i+1,j]}&\text{if }i+1\not\equiv j\text{ mod }n\\ (1-[\C_{\varepsilon_i-\varepsilon_1}])\cO^{[i,j]}+[\C_{\varepsilon_i-\varepsilon_1}](\cO^{[i+1,j-1]}+\cO^{[i+2,j]}-\cO^{[i+2,j-1]})&\text{if }i+1\equiv j\text{ mod }n\end{cases},\end{equation*}
    \begin{equation*}\cO^{[i,j]}\star\cO^{[2]}=\begin{cases}(1-[\C_{\varepsilon_n-\varepsilon_j}])\cO^{[i,j]}+[\C_{\varepsilon_n-\varepsilon_j}]\cO^{[i,j-1]}&\text{if }i+1\not\equiv j\text{ mod }n\\ (1-[\C_{\varepsilon_n-\varepsilon_j}])\cO^{[i,j]}+[\C_{\varepsilon_n-\varepsilon_j}](\cO^{[i+1,j-1]}+\cO^{[i,j-2]}-\cO^{[i+1,j-2]})&\text{if }i+1\equiv j\text{ mod }n\end{cases}.\end{equation*}
\end{thm}
A Chevalley formula for the equivariant quantum \(K\)-theory ring of the complete flag variety \(G/B\) has been proved by Lenart, Naito, and Sagaki \cite{lenart}. Kato has established a ring homomorphism from the equivariant quantum \(K\)-theory ring of \(G/B\) to that of \(G/P\) \cite{kato}. \Theorem{chev} can also be proved by combining these results and working out the cancellations combinatorially. This approach is taken by Kouno, Lenart, Naito, and Sagaki in \cite{kouno2021}, where they prove Chevalley formulas for the equivariant quantum \(K\)-theory ring of Grassmannians of type \(A\) and \(C\).\footnote{The paper \cite{kouno2021} was updated to also include Chevalley formulas for two-step flag varieties of type \(A\).} 

We also obtain the following closed formula for Littlewood--Richardson coefficients for the non-equivariant quantum \(K\)-theory ring of \(X\). 
\begin{thm}[\(=\) \Theorem{LR}]
    In \(QK(X)\), for \([i,j],\ [k,l]\in W^P\),\begin{equation*}
        \cO^{[i,j]}\star\cO^{[k,l]}=
        \begin{cases}
            \cO^{[x,y]}&\text{if } x-y<n[\chi(i>j)+\chi(k>l)]\\ \cO^{[x,y-1]}+\cO^{[x+1,y]}-\cO^{[x+1,y-1]}&\text{if } x-y\geq n[\chi(i>j)+\chi(k>l)]
        \end{cases},
    \end{equation*}where $x=i+k-1$, $y=j+l-n$.
\end{thm}

As an application, we prove that for general \(g\in G\) and \(u,\ v,\ w\in W^P\), the Gromov--Witten variety \[M_d(g.X^u, X^v, X_w):=\ev_1^{-1}(g.X^u)\cap\ev_2^{-1}(X^v)\cap\ev_3^{-1}(X_w)\subseteq \M_{0,3}(X,d)\] has arithmetic genus \(0\) whenever it is non-empty (see \Corollary{genus0}).

Positivity of Schubert structure constants has been a topic of interest in the study of (quantum) Schubert calculus \cites{Brion,graham,anderson2010positivity,mihalcea,anderson2011positivity,qkpos}. The above-mentioned formulas that we derive are positive in the appropriate sense (see \Corollary{chevpos} and \Corollary{LRpos}).

There have also been studies on the powers of the quantum parameters \(q\) that appear with nonzero coefficients in a product of Schubert classes in the quantum cohomology and quantum \(K\)-theory rings of homogeneous spaces \cites{fulton2004quantum,Euler,Postnikov, lenart,qkpos,shifler2022minimum}. Our Littlewood--Richardson rule for \(\Fl(1,n-1;n)\) shows that the powers of \(q\) that appear with nonzero coefficients in \(\cO^u\star\cO^v\) form an interval between two degrees (see \Corollary{power}). 

Using the ring homomorphism from the quantum \(K\)-theory ring of \(X\) to that of \(\P^{n-1}\) supplied by the aforementioned result of Kato \cite{kato}, we recover known formulas for projective spaces \cites{qkchev, kouno2020, kouno2021, qkgrass, Chung} (see \Corollary{chevP^n} and \Corollary{LRP^n}). In a different direction, from our formulas one can easily derive analogous formulas for the (equivariant) (quantum) cohomology ring of \(X\). Chevalley formulas are known for the (equivariant) (quantum) cohomology rings of all \(G/P\) \cite{mihalcea2007equivariant} (see also \cite{fulton2004quantum} for the non-equivariant case and \cite{ciocan1999quantum} for the non-equivariant case in type \(A\)), but Littlewood--Richardson rules for two-step flag varieties are only known in the classical case \cites{coskun2009littlewood,buch2015mutations, buch2016puzzle,knutson2017schubert}.

Our approach for proving \Theorem{geom} is as follows. It suffices to show that when \(p\) is a general point in \(\Gamma_d(X_u,D)\), the Gromov--Witten variety \(M_d(X_u,D,p)\) is rationally connected. By a result of Behrend and Manin \cite{behrend.manin:stacks}, the projection \(p_1: X\to\P^{n-1}\) induces a morphism \[\rho: \M_{0,3}(X,d)\to \M_{0,3}(\P^{n-1},{p_1}_*d).\] We restrict \(\rho\) to a morphism from \(M_d(X_u,D,p)\) to the corresponding Gromov--Witten variety for \(\P^{n-1}\). While the restriction is not surjective in general, we use an involution of \(X\) to reduce to a situation where it is surjective (see \Corollary{surj3}). Together with the irreducibility of \(M_d(X_u,D,p)\) (see \Proposition{irred}) and parametrizations of the target and the general fibre, a result of Graber, Harris, and Starr \cite{GHS} implies that \(M_d(X_u,D,p)\) is rationally connected. 

This paper is organized as follows. In \Section{prelim}, we set up notations for working with the Kontsevich moduli space and the incidence variety and review some known geometric results that will be used in later sections. In \Section{geom}, we study the geometry of Gromov--Witten varieties and prove \Theorem{geom}. In \Section{chev} we apply \eqn{qclassical} to compute \(K\)-theoretic Gromov--Witten invariants, derive the Chevalley formula for the equivariant quantum \(K\)-theory ring of \(X\) and discuss some immediate consequences. \Section{LR} is devoted to the quantum \(K\)-theoretic Littlewood--Richardson rule and its implications.

While writing this paper, the author learned through private communications that Sybille Rosset proved independently in her doctoral thesis \cite{rosset2022quantum} \Theorem{geom}, the nonequivariant case of \Theorem{chev}, and the positivity of all structure constants of \(QK(X)\). 

\subsection*{Acknowledgements} I would like to thank my advisor Anders Buch for introducing me to this area and for many helpful conversations. I also thank Rahul Pandharipande for answering a question related to \Section{func} and Nicolas Perrin for helpful comments on the last version of this paper and for asking a question that led to the discovery of \Corollary{genus0}. I appreciate the helpful comments given by the anonymous referees. Part of this project was done while I virtually attended the ICERM Combinatorial Algebraic Geometry semester program. I thank ICERM and the organizers of this program for an inspiring semester.  

\section{Preliminaries}\label{sec:prelim}

Varieties are not assumed to be irreducible, but rational, unirational, and rationally connected varieties are. Let \(f:X\to Y\) be a morphism of varieties, then the general fibre is a fibre \(f^{-1}(u)\), where \(u\) is a closed point chosen from a suitable dense open subset \(U\subseteq Y\). 

\subsection{Moduli Space of Stable Maps}
\subsubsection{Overview} 
Given a smooth complex projective variety \(X\) and an effective degree \[d\in H_2(X)^+\coloneqq\{a\in H_2(X,\Z): (a,D)\geq 0\ \forall\text{ ample }D\in H^2(X,\Z)\},\] let \(M_{g,m}(X,d)\) be the moduli space of maps \(f\) from a connected projective nonsingular curve \(C\) of genus \(g\) with \(m\) distinct marked points to \(X\) such that \(f_*[C]=d\). 
There is a compactification \[M_{g,m}(X,d)\subseteq \M_{g,m}(X,d),\] whose points correspond to isomorphism classes of \(m\)-pointed stable maps \(f:C\to X\) representing the class \(d\), i.e., \(f_*[C]=d\), where \(C\) is allowed to be an \(m\)-pointed connected nodal curve of arithmetic genus \(g\). Evaluating each \(f\) at the \(i\)-th marked point gives a morphism \[\ev_i: \M_{g,m}(X,d)\to X.\] We write \(\ev\) for \[\ev_1\times\dots\times \ev_m: \M_{g,m}(X,d)\to X^m.\] 
The contravariant functor \(\Mb_{g,m}(X,d)\) from the category of complex schemes to sets is defined as follows: given a complex scheme \(S\), \(\Mb_{g,m}(X,d)(S)\) is the set of isomorphism classes of stable families over \(S\) of maps from \(m\)-pointed curves of genus \(g\) to \(X\) representing the class \(d\). \(\M_{g,m}(X,d)\) is a projective coarse moduli space for \(\Mb_{g,m}(X,d)\).
 
See \cite{fulton1997notes} and \cite{behrend.manin:stacks} for more details. 
 
\subsubsection{Functoriality}\label{sec:func}
 The goal of this section is to specialize 
  \cite[Theorem 3.6]{behrend.manin:stacks} to the following.
\begin{prop}\label{prop:nat}
Given a morphism of smooth projective varieties \(\xi: Y\to Z\) and \(d\in H_2(Y)^+\), there exists a natural transformation \[\Mb_{g,m}(Y,d)\to\Mb_{g,m}(Z,\xi_*d)\] given by composing stable families with \(\xi\) and collapsing components of source curves in families as necessary to make the compositions stable.
\end{prop}
\begin{proof}
We adopt the definitions in \cite{behrend.manin:stacks}. 

Let \(\tau\) be the modular graph with one vertex \(v\) of genus \(g\) and \(m\) tails. Let \(\alpha:\{v\}\to H_2(Y)^+,\ \beta:\{v\}\to H_2(Z)^+\) be given by \(\alpha(v)=d\), \(\beta(v)=\xi_*d\). Then \((\tau,\alpha)\) and \((\tau,\beta)\) are stable \(H_2(Y)^+\)-graph and \(H_2(Z)^+\)-graph, respectively, and the quadruple \((\xi, id,\tau,id)\) is a morphism from \((Y,\tau,\alpha)\) to \((Z,\tau,\beta)\) in \(\mathfrak{BG}_s\). 

Let \(S\) be a complex scheme. A stable map \((C,x,f)\) over \(S\) from an \(m\)-pointed curve (of genus \(g\)) to \(Y\) (of class \(d\)) is the same as a stable \((Y,\tau,\alpha)\)-map over \(S\).

By \cite[Theorem 3.6]{behrend.manin:stacks}, there exists a pushforward \((D,y,h)\) of \((C,x,f)\) of stable maps under \((\xi,id,\tau,id)\). Since \((D,y,h)\) is a stable map over \(S\) from an \(m\)-pointed curve of genus \(g\) to \(Z\) of class \(\xi_*d\), this pushforward defines a map \(\rho_S: \Mb_{g,m}(Y,d)(S)\to\Mb_{g,m}(Z,\xi_*d)(S)\). Moreover, this pushforward commutes with base change, making \(\rho\) a natural transformation. The proof of \cite[Theorem 3.6]{behrend.manin:stacks} shows that the pushforward is constructed by composing stable families with \(\xi\) and collapsing components of the source curves in families as necessary to make the compositions stable.
\end{proof}

\begin{cor}\label{cor:map}
Given a morphism of smooth projective varieties \(\xi: Y\to Z\) and \(d\in H_2(Y)^+\), there exists a morphism \[\M_{g,m}(Y,d)\to\M_{g,m}(Z,\xi_*d)\] given by composing stable maps with \(\xi\) and collapsing components of source curves as necessary to make the compositions stable.
\end{cor}

\begin{proof}
This follows from \Proposition{nat} and \cite[Theorem 1]{fulton1997notes}.
\end{proof}

\subsubsection{Stable Maps to Flag Varieties}

We are mainly interested in the case where \(g=0\) and \(X\) is a flag variety \(G/P\) defined by a complex connected semisimple linear algebraic group \(G\) and a parabolic subgroup \(P\). Schubert varieties in \(X\) are orbit closures for the action of a Borel subgroup \(B\subseteq G\). We write \(B^-\) for the Borel subgroup conjugate to \(B\) by the longest Weyl group element.

Let \(X_1,\dots,X_l\) be Schubert varieties in \(X\) with \(l\leq 3\) (not necessarily in general position). We will be concerned with \[M_d(X_1,\dots,X_l)\coloneqq \ev_{1}^{-1}(X_1)\cap\dots\cap \ev_{l}^{-1}(X_l)\subseteq \M_{0,3}(X,d).\] When \(X_1\) is \(B\)-stable and \(X_2\) is \(B^-\)-stable, \(M_d(X_1,X_2)\) is either empty or unirational \cite[Proposition 3.2]{Buch2013}. We will also work with \[\mathring{M}_d(X_1,\dots,X_l)\coloneqq M_{0,3}(X,d)\cap M_d(X_1,\dots,X_l),\] the degree \(d\) curve neighborhood of \(X_1\) \[\Gamma_d(X_1)\coloneqq \ev_2(M_d(X_1)),\] i.e., the closure of the union of degree \(d\) rational curves (images of degree \(d\) maps from \(\P^1\)) in \(X\) that meet \(X_1\), and  \[\Gamma_d(X_1,X_2)\coloneqq \ev_3(M_d(X_1, X_2)),\] which is the closure of the union of degree \(d\) rational curves in \(X\) connecting \(X_1\) and \(X_2\). 

Using automorphisms of \(\P^1\), we can think of closed points in \(M_{0,3}(X,d)\) as maps \(\P^1\to X\) of degree \(d\) with marked points at \(0,\ 1,\ \infty\in\P^1\).

\subsection{Incidence Varieties}

\subsubsection{Definition and Basic Properties}\label{sec:def}

From now on \(n\geq3\). \begin{equation*}X=\Fl(1,n-1;n)=\{\text{flags of vector spaces }U\subset V\subset\C^n:\ \dim U=1,\dim V=n-1\}\end{equation*} has dimension \(2n-3\). Let \[p_1: X\to\Gr(1,n)\cong\P(\C^n)\text{ and }p_2: X\to\Gr(n-1,n)\cong\P({\C^n}^*)\] be the projections.  Schubert varieties in \(X\) are indexed by pairs of unequal integers \([i,j]\) (we use square brackets to distinguish them from degrees), where \(i\) and \(j\) are each between \(1\) and \(n\). We denote this indexing set by \(W^P\). $X$ embedds in $\P(\C^n)\times\P({\C^n}^*)$ with image given by \[x_1y_1+\dots+x_n y_n=0,\] where $x_1,\dots, x_n$ are the standard coordinates on $\C^n$ and $y_1,\dots, y_n$ are the dual coordinates on ${\C^n}^*$. 
Under this identification, the Schubert variety \(X_{[i,j]}\) is the subvariety given by \[x_{i+1}=\dots=x_n=y_1=\dots=y_{j-1}=0,\] and the opposite Schubert variety $X^{[i,j]}$ is the subvariety given by \[x_1=\dots=x_{i-1}=y_{j+1}=\dots=y_n=0.\] 
Note that \[\dim X_{[i,j]}=\text{codim}X^{[i,j]}=(i-1)+(n-j)-\chi(i>j),\] where \[\chi(i>j)=\begin{cases}1&\text{if }i>j\\0&\text{if }i\leq j\end{cases}.\] The two Schubert divisors \[D^{[1]}\coloneqq X^{[2,n]}\text{ and }D^{[2]}\coloneqq X^{[1,n-1]}\] are cut out by \(x_1=0\) and \(y_n=0\), respectively.
The involution of $\P(\C^n)\times\P({\C^n}^*)$ given by swapping $x_i$ and $y_{n+1-i}$ induces an involution \(\varphi\) of $X$.
\(\varphi\) swaps $X^{[i,j]}$ and $X^{\iota[i,j]}$, where \[\iota[i,j]:=[n+1-j,n+1-i].\] In particular, \(\varphi\) swap \(D^{[1]}\) and \(D^{[2]}\). An effective degree of \(X\) is a pair of non-negative integers \(d=(d_1,d_2)\), which is identified with \(d_1[X_{[2,n]}]+d_2[X_{[1,n-1]}]\in H_2(X)^+\). 

Let \(G=\SL(\C^n)\) (of type \(A_{n-1}\)), \(T\subset G\) be the maximal torus of diagonal matrices and \(B\subset G\) be the Borel subgroup of upper triangular matrices.
For \(i=1,\dots,n\), let \(\varepsilon_i: T\to \C^*\) be the character that sends a diagonal matrix to its \(i\)-th diagonal entry. Then the root system of \((G,B,T)\) is \[\Phi=\{\alpha_{i,j}\coloneqq\varepsilon_i-\varepsilon_j: i\neq j, 1\leq i,j\leq n\},\] and the
subset of simple roots is \[\Delta=\{\alpha_{1,2}, \alpha_{2,3}, \dots, \alpha_{n-1,n}\}\subset\Phi.\] The Weyl group \(W\) consists of permutations of  \(\{\varepsilon_1,\dots,\varepsilon_n\}\), and can be identified with \(S_n\). Let \[\Delta_P=\{\alpha_{2,3},\dots,\alpha_{n-2,n-1}\}\subset\Delta.\] \(\Delta_P\) determines a parabolic subgroup \(P\), which consists of block upper triangular matrices with \(3\) blocks sized \(1,\ n-2,\ 1\). $W_P$ is the subgroup of \(W\) consisting of permutations of \(\{\varepsilon_2,\dots,\varepsilon_{n-1}\}\). We identify \(W^P\) with the set of coset representatives of $W/W_P$ of minimal length by sending \([i,j]\) to the unique permutation \(w\) such that \(w(1)=i,\ w(n)=j\), and \(w(2)<\dots<w(n-1)\). For \(u\in W\), write \(\ell(u)\) for the length of \(u\). The Bruhat order on \(W\) gives \(W^P\) a poset structure. The set of closed points in \(X\) can be identified with the quotient set \(G/P\) by the map \begin{equation}\label{eqn:id}
    g.(E_1, E_{n-1})\mapsto gP,
\end{equation} where \(E_i=\text{Span}\{\mathbf{e}_1,\dots,\mathbf{e}_i\}\) is the span of the first \(i\) standard basis vectors in \(\C^n\). In \(\P(\C^n)\times\P({\C^n}^*)\), \(gP\) has coordinates \[x_i=g_{i,1},\  y_i={C(g)}_{i,n},\] where \(C(g)\) is the cofactor matrix of \(g\). Let \(\tilde\varphi: G\to G\) be defined as \[{\tilde\varphi(g)}_{i,j}={g^{-1}}_{n+1-j,n+1-i}={C(g)}_{n+1-i,n+1-j}\ \forall g\in G.\] Then \(\tilde\varphi\) is an involution in the category of algebraic groups. Moreover, \(\tilde\varphi\) restricts to an involution of \(P\) and an involution of \(T\) which sends \(\text{diag}(t_1,\dots,t_n)\) to diag\(({t_n}^{-1},\dots,{t_1}^{-1})\). In particular, \(\tilde\varphi\) induces an involution \(\varphi\) of \(G/P\). Under our identifications, this is exactly the \(\varphi\) defined previously.

\subsubsection{Curve Neighborhood}\label{sec:curvenbhd}

Clearly \[\Gamma_0(X_{[i,j]})=X_{[i,j]}.\] 

When \(d_1>0\) and \(d_2=0\), since the image of every stable map of degree \(d\) is contained in a fibre of \(p_2\), and every two points in a fibre of \(p_2\) are connected by a line and therefore the image of a stable map of degree \(d\), \[\Gamma_{d}(X_{[i,j]})=p_2^{-1}(p_2(X_{[i,j]}))=\begin{cases}X_{[n,j]}&\text{if }j<n\\X_{[n-1,n]}&\text{if }j=n\end{cases}.\] 

Similarly, when \(d_1=0\) and \(d_2>0\), \[\Gamma_{d}(X_{[i,j]})=p_1^{-1}(p_1(X_{[i,j]}))=\begin{cases}X_{[i,1]}&\text{if }i>1\\X_{[1,2]}&\text{if }i=1\end{cases}.\] 

When \(d_1,\ d_2>0\), we claim that \(\Gamma_d(X_{[i,j]})=X\). Since \[\Gamma_{(1,1)}(X_{1,n})\subseteq\Gamma_d(X_{[i,j]})\subseteq X,\] it suffices to show \(\Gamma_{(1,1)}(X_{1,n})=X\). \(X_{1,n}=\{A\}\), where \(A=(A_1,A_2)\) with \(A_1=\text{Span}\{\mathbf{e}_1\}\), \(A_2=\text{Span}\{\mathbf{e}_1,\dots,\mathbf{e}_{n-1}\}\). Let \(U\subseteq X\) be the open dense subset \[\{(B_1,B_2)\in X:\ B_1\nsubseteq A_2,\  B_2\nsupseteq A_1\}.\] Given \(B=(B_1,B_2)\in U\), write \(B_1=\text{Span}\{\mathbf{b}_1\}\), and let \(f:\P^1\to X\) be given by \[(s:t)\mapsto (\text{Span}\{t\mathbf{e}_1+s\mathbf{b}_1\},\text{Span}\{t\mathbf{e}_1+s\mathbf{b}_1\}+A_2\cap B_2).\] Then \(f\) has degree \((1,1)\), and \(f(0)=A\), \(f(\infty)=B\). This shows \(\Gamma_{(1,1)}(X_{1,n})\supseteq U\). Since \(\Gamma_{(1,1)}(X_{1,n})\) is closed, it must be the entire \(X\).

By the same reasoning, \[\Gamma_0(X^{[i,j]})=X^{[i,j]};\] when \(d_1>0\) and \(d_2=0\), \[\Gamma_{d}(X^{[i,j]})=\begin{cases}X^{[1,j]}&\text{if }j>1\\X^{[2,1]}&\text{if }j=1\end{cases};\] when \(d_1=0\) and \(d_2>0\), \[\Gamma_{d}(X^{[i,j]})=\begin{cases}X^{[i,n]}&\text{if }i<n\\X^{[n,n-1]}&\text{if }i=n\end{cases};\] when \(d_1>0\) and \(d_2>0\), \[\Gamma_d(X^{[i,j]})=X.\]

\begin{lemma}\label{lemma:gamma}
\(\Gamma_d(X_{[i,j]},D^{[k]})=\begin{cases}\Gamma_d(X_{[i,j]})&\text{if }d_k>0\\ \Gamma_d(X_{[i,j]})\cap D^{[k]}&\text{if }d_k=0.\end{cases}\)
\end{lemma}

\begin{proof}
\(\Gamma_d(X_{[i,j]},D^{[k]})\) is the closure of the union of degree \(d\) rational curves in \(X\) meeting \(X_{[i,j]}\) and \(D^{[k]}\). 

When \(d_k>0\), every degree \(d\) rational curve in \(X\) meets \(D^{[k]}\), so \(\Gamma_d(X_{[i,j]},D^{[k]})\) is the closure of the union of degree \(d\) rational curves in \(X\) meeting \(X_{[i,j]}\), which is \(\Gamma_d(X_{[i,j]})\). 

When \(d_k=0\), let \(C\) be a degree \(d\) rational curve in \(X\), then \(C\cap D^{[k]}\) is either empty or \(C\). In other words, \(C\) meets \(D^{[k]}\) if and only if \(C\subseteq D^{[k]}\).
\end{proof}

\begin{cor}\label{cor:ratsing}
\(\Gamma_d(X_{[i,j]},D^{[k]})\) has rational singularities.
\end{cor}
\begin{proof}
\Lemma{gamma} shows that \(\Gamma_d(X_{[i,j]},D^{[k]})\) is a Richardson variety, i.e., an intersection of two opposite Schubert varieties. Schubert varieties are known to have rational singularities \cite{Ramanathan}. It then follows from \cite[Lemma 2]{Brion} that Richardson varieties have rational singularities. 
\end{proof}

\subsection{Helpful Lemmas}

The following lemmas will be used in the proofs of \Proposition{irred} and \Theorem{ratcon}.

\begin{lemma}\label{lemma:fibrecomp}
Let \(\phi:Y\to Z\) be a dominant morphism between irreducible varieties and \(U\) be a non-empty open subset of \(Y\). Then every irreducible component of the general fibre meets \(U\).
\end{lemma}

\begin{proof}
Let \(W\) be an irreducible component of \(Y\setminus U\). If \(\phi(W)\) is not dense in \(Z\), then the general fibre avoids \(W\). If \(\phi(W)\) is dense in \(Z\), then by generic flatness, the irreducible components of the intersection of the general fibre with \(W\) have strictly smaller dimension than do irreducible components of the general fibre. Therefore, no irreducible component of the general fibre is contained in \(W\).
\end{proof}

\begin{lemma}\label{lemma:reduced}
Let \(\phi:Y\to Z\) be a dominant morphism between irreducible varieties over a field of characteristic \(0\). Then the general fibre is reduced.
\end{lemma}

\begin{proof}

Without loss of generality, assume \(Z=\text{Spec}A\), \(Y=\text{Spec}B\). Then the generic fibre is \(\text{Spec}(B\otimes_A A_0)\). \(\phi\) being dominant implies \(A\subseteq B\). So \(B\otimes_A A_0\subseteq B_0\) is a domain. Since the characteristic is \(0\), the generic fibre is geometrically reduced by \cite[\href{https://stacks.math.columbia.edu/tag/020I}{Lemma 020I}]{stacks-project}, and the conclusion follows from \cite[\href{https://stacks.math.columbia.edu/tag/0578}{Lemma 0578}]{stacks-project}.
\end{proof}

\section{Geometry of Gromov--Witten Varieties}\label{sec:geom}

By \Corollary{map}, there is a map \(\rho: \M_{0,3}(X,d)\to \M_{0,3}(\P^{n-1},d_1)\) given by composing with \(p_1\) and collapsing components of the source curve as necessary to make the composition stable. In particular, the diagram     \[\begin{CD}\overline{M}_{0,3}(X,d)     @>\rho>>  \M_{0,3}(\P^{n-1},d_1)\\@VV\ev V @VV\ev V       \\ X^3 @>{p_1}^3>> (\P^{n-1})^3\end{CD}\] commutes. Let \[\rho_2: M_d(X_{[i,j]},X^{[n,1]})\to M_{d_1}(p_1(X_{[i,j]}),p_1(X^{[n,1]}))\] and \[\rho_3^{g,k}: M_d(X_{[i,j]},X^{[n,1]},g.D^{[k]})\to M_{d_1}(p_1(X_{[i,j]}),p_1(X^{[n,1]}),g.p_1(D^{[k]}))\] be restrictions of \(\rho\), where \(g\in G\) is not assumed to be general. Using the automorphism \(\varphi\), we can often reduce to the case where \begin{equation}\label{eqn:assumption}\text{either }d_1<d_2,\text{ or }d_1=d_2\text{ and }k=1.\end{equation} 

\begin{lemma}\label{lemma:surjpts}
Assume \(d_1\leq d_2\). Let \(A=(A_1,A_2),\ B=(B_1,B_2)\in X\) be pairs of vector spaces such that \(A_1\nsubseteq B_2\), \(B_1\nsubseteq A_2\). Then the restriction \[\rho: M_d(A,B)\to M_{d_1}(A_1,B_1)\] is surjective.
\end{lemma}

\begin{proof}
Let \(f_1:\P^1\to\Gr(1,n)\) be a map of degree \(d_1\) such that \[f_1(0)=A_1,\ f_1(\infty)=B_1.\] It suffices to show \(f_1\) is in the image.

Let \[U=\{a\in\P^1: f_1(a)\nsubseteq A_2\cap B_2\}.\] Since \(0,\infty\in U\), \(U\) is non-empty. Being the preimage of the complement of a linear subspace, \(U\) is open. Therefore, \(U\) is dense in \(\P^1\). Since \(\P^1\) is nonsingular and \(X\) is projective, the map \(\tilde{f}: U\to X\) given by \[\tilde{f}(a)=(f_1(a),f_1(a)+(A_2\cap B_2))\] can be extended to all of \(\P^1\). The extended map \(\tilde{f}\) has degree \((d_1,d_2')\) for some \(d_2'\leq d_1\). Indeed, \(d_2'\) counts the number of \(a\)'s in \(U\) such that 
\begin{equation}\label{eqn:condition}
    f_1(a)+(A_2\cap B_2)\text{ contains a general }1\text{ dimensional subspace }C_1\text{ of }\C^n.
\end{equation}  \eqn{condition} is equivalent to \[f_1(a)+(A_2\cap B_2)=C_1+(A_2\cap B_2),\] which is in turn equivalent to the condition that \[f_1(a)\text{ is contained in the hyperplane }C_1+(A_2\cap B_2).\] Therefore, \(d_2'\leq d_1\) by the definition of \(d_1\). By attaching additional components to the domain as necessary using \cite[Lemma 12]{fulton1997notes}, we get a stable map \(f:C\to X\) of degree \(d\) with image \(f_1\). Note that the added components collapse after composing with \(p_1\).
\end{proof}

\begin{lemma}\label{lemma:surj2}
When \(d_1\leq d_2\) and \(M_d(X_{[i,j]},X^{[n,1]})\) is non-empty, \[\rho_2:M_d(X_{[i,j]},X^{[n,1]})\to M_{d_1}(p_1(X_{[i,j]}),p_1(X^{[n,1]}))\] is surjective.  
\end{lemma}

\begin{proof}
When \(d_1=0\), \(M_0(p_1(X_{[i,j]}),p_1(X^{[n,1]}))\) is non-empty if and only if \(i=n\) and in this case it is a point. \[\rho_2:M_d(X_{[n,j]},X^{[n,1]})\to M_0(p_1(X_{[n,j]}),p_1(X^{[n,1]}))\] is surjective if and only \(M_d(X_{[n,j]},X^{[n,1]})\) is non-empty.

Now assume \(d_1\geq1\). 

\(X^{[n,1]}=\{(B_1,B_2)\}\), where we think of \(B_1,\ B_2\) as \[B_1=(0:\dots:0:1)\in\P(\C^n),\ B_2=(1:0:\dots:0)\in\P({\C^n}^*).\]
\(\ev_1: \M_{0,3}(\P(\C^n),d_1)\to\P(\C^n)\) maps \(\ev_2^{-1}(B_1)\) surjectively to \(\P(\C^n)\) and \(M_{d_1}(p_1(X_{[i,j]}),B_1)\) surjectively to \(p_1(X_{[i,j]})\). 

We write \(D_+(s)\) and \(V_+(s)\) for the non-vanishing and vanishing sets in \(\P(\C^n)\) for the function \(s\), respectively.

If \(i<n\), let \[U\coloneqq\ev_1^{-1}(D_+(x_1))\cap \mathring{M}_{d_1}(p_1(X_{[i,j]}),B_1);\] if \(i=n\), let \[U\coloneqq\ev_1^{-1}(D_+(x_j))\cap \ev_1^{-1}(D_+(x_1))\cap \mathring{M}_{d_1}(p_1(X_{[i,j]}),B_1).\] Note that \(p_1(X_{[i,j]})\) is cut out by \(x_{i+1}=\dots=x_n=0\), so \(U\) is non-empty and therefore dense in \(M_{d_1}(p_1(X_{[i,j]}),B_1)\). Let \(f_1\in U\). It suffices to show \(f_1\) is in the image.

Let \(A_1=(a_1:\dots:a_i:0:\dots:0)\in\P(\C^n)\) be \(\ev_1(f_1)\). Then \(a_1\neq0\), so \(A_1\nsubseteq B_2\). \(p_2(p_1^{-1}(A_1)\cap X_{[i,j]})\) is the set \[\{(y_1:\dots:y_n)\in\P({\C^n}^*): a_1 y_1+\dots+a_i y_i=0, y_1=\dots=y_{j-1}=0\}\] and  \[p_2(p_1^{-1}(B_1))=V_+(y_n).\] One checks that \[p_2(p_1^{-1}(A_1)\cap X_{[i,j]})\nsubseteq p_2(p_1^{-1}(B_1))\] (when \(i=n\) this is guaranteed by \(a_j\neq0\)). Therefore, there exists \(A_2\in p_2(X_{[i,j]})\) such that \(A_1\subseteq A_2\) and \(B_1\nsubseteq A_2\). Let \(A=(A_1,A_2)\), \(B=(B_1,B_2)\). By \Lemma{surjpts}, there exists \(f\in M_d(A,B)\) such that \(\rho_2(f)=f_1\). Since \(A\in X_{[i,j]}\), \(B\in X^{[n,1]}\), we have \(M_d(A,B)\subseteq M_d(X_{[i,j]},X^{[n,1]})\) and the conclusion follows. 
\end{proof}

\begin{cor}\label{cor:surj3}
    Assume \eqn{assumption} is satisfied, and \(M_d(X_{[i,j]},X^{[n,1]},g.D^{[k]})\) is non-empty. Then \[\rho_3^{g,k}: M_d(X_{[i,j]},X^{[n,1]},g.D^{[k]})\to M_{d_1}(p_1(X_{[i,j]}),p_1(X^{[n,1]}),g.p_1(D^{[k]}))\] is surjective.
\end{cor}

\begin{proof}
Non-emptiness of \(M_d(X_{[i,j]},X^{[n,1]},g.D^{[k]})\) implies non-emptiness of \(M_d(X_{[i,j]},X^{[n,1]})\).

When \(k=1\), since \(D^{[1]}=p_1^{-1}(p_1(D^{[1]}))\), each fibre of \(\rho_3^{g,1}\) is a fibre of \(\rho_2\). Surjectivity of \(\rho_2\) implies surjectivity of \(\rho_3^{g,1}\).

When \(k=2\), \eqn{assumption} implies \(d_1<d_2\). Let \[f_1\in M_{d_1}(p_1(X_{[i,j]}),p_1(X^{[n,1]}),g.p_1(D^{[2]}))=M_{d_1}(p_1(X_{[i,j]}),p_1(X^{[n,1]})).\] 

By \Lemma{surj2}, there exists \(\tilde{f}: \tilde{C}\to X\) in \(M_{(d_1,d_2-1)}(X_{[i,j]},X^{[n,1]})\) mapping to \(f_1: C\to X\). 

Let \(a\in\tilde{C}\) be the 3rd marked point and \(A=(A_1,A_2)=\tilde{f}(a)\). If \(A\in g.D^{[2]}\), by \cite[Lemma 12]{fulton1997notes} we can attach an additional component to \(\tilde{C}\) to obtain \(f\in M_d(X_{[i,j]},X^{[n,1]},g.D^{[2]})\). Otherwise, pick \(A_2'\in g.p_2(D^{[2]})\) such that \(A_1\subseteq A_2'\).  Let \(v\in A_2\setminus A_2'\), \(v'\in A_2'\setminus A_2\). Attach \(\P^1\) to \(\tilde{C}\) by identifying \((1:0)\) with \(a\) and map this copy of \(\P^1\) to \(X\) by \[(s:t)\mapsto (A_1, A_2\cap A_2'+\text{Span}\{s v+t v'\}).\] Replace the 3rd marked point with \((0:1)\) on the attached \(\P^1\) component. \cite[Lemma 12]{fulton1997notes} implies we get a map \(f\in M_d(X_{[i,j]},X^{[n,1]},g.D^{[k]})\). 

Note that the added component, among others, collapses after composing with \(p_1\), so \(\rho(f)=f_1\).
\end{proof}
 
\begin{lemma}\label{lemma:param} 
Let \(X_1,\ X_2,\ X_3\) be Schubert varieties in \(X\) (not necessarily in general position). Let \[\mathring{\rho}: \mathring{M}_d(X_1,X_2,X_3)\to \mathring{M}_{d_1}(p_1(X_1),p_1(X_2),p_1(X_3))\] be the restriction of \(\rho\). Let \(f_1\in \mathring{M}_{d_1}(p_1(X_1),p_2(X_2),p_1(X_3))\). Then \(\mathring{\rho}^{-1}(f_1)\) is rational whenever non-empty.
\end{lemma}

\begin{proof}
\(f_1: \P^1\to\Gr(1,n)\) is of the form \[(s:t)\mapsto\text{Span}\{\sum_{i=0}^{d_1} \mathbf{a}_i s^i t^{d_1-i}\},\] where \(\mathbf{a}_i=(a_{i1},\dots, a_{in})\in\C^n\). Let \(f\in\mathring{\rho}^{-1}(f_1)\), then \(f\) is uniquely determined by \(f_2\coloneqq p_2\circ f\), where \[f_2((s:t))=\text{Span}\{\sum_{i=0}^{d_2} \mathbf{b}_i s^i t^{d_2-i}\}\] for some \(\mathbf{b}_0,\dots, \mathbf{b}_{d_2}\in{\C^n}^*\) such that \[\sum_{i=0}^{d_2} \mathbf{b}_i s^i t^{d_2-i}\neq \mathbf{0}\text{ for all }(s:t)\in\P^1,\] \[(\sum_{i=0}^{d_1} \mathbf{a}_i s^i t^{d_1-i})\cdot(\sum_{i=0}^{d_2} \mathbf{b}_i s^i t^{d_2-i})=0\text{ for all }(s:t)\in\P^1,\] \[\text{Span}\{\mathbf{b}_{0}\}\in p_2(X_1),\ \text{Span}\{\sum_{i=0}^{d_2}\mathbf{b}_{i}\}\in p_2(X_2),\ \text{Span}\{\mathbf{b}_{d_2}\}\in p_2(X_3).\] This parametrizes \(\mathring{\rho}^{-1}(f_1)\) as an open subset of the intersection of some hyperplanes in \(\P^{n(d_2+1)-1}\).
\end{proof}

\begin{lemma}\label{lemma:rational}
Let \(Y_1,Y_2,Y_3\) be linear subspaces in general position in \(\P^n\). Then \(M_e(Y_1,Y_2,Y_3)=\ev_1^{-1}(Y_1)\cap \ev_2^{-1}(Y_2)\cap \ev_3^{-1}(Y_3)\subseteq\M_{0,3}(\P^n,e)\) is either empty or rational.   
\end{lemma}

\begin{proof}
When \(M_e(Y_1,Y_2,Y_3)\) is non-empty, by \cite[Theorem 2 and Remark 7]{kleiman:transversality}, it is reduced and every irreducible component meets \(M_{0,3}(\P^n,e)\). On the other hand, \(\mathring{M}_e(Y_1,Y_2,Y_3)=M_e(Y_1,Y_2,Y_3)\cap M_{0,3}(\P^n,e)\) is rational by an argument similar to the proof of \Lemma{param}. It follows that \(M_e(Y_1,Y_2,Y_3)\) is rational.
\end{proof}

We thank an anonymous referee for pointing out the following Proposition was proved independently in \cite{rosset2022quantum}.

\begin{prop}\label{prop:irred}
 For general \(g\), \(M_d(X_{[i,j]},X^{[n,1]},g.D^{[k]})\) is irreducible whenever non-empty.  
\end{prop}

\begin{proof}
Note that \(M_0(X_{[i,j]},X^{[n,1]},g.D^{[k]})\cong X_{[i,j]}\cap X^{[n,1]}\cap g.D^{[k]}=\varnothing\).

When \(d>0\), without loss of generality, we assume \eqn{assumption}.

By \cite[Remark 7]{kleiman:transversality}, \(M_d(X_{[i,j]},X^{[n,1]},g.D^{[k]})\) is normal. It then follows from  \cite[\href{https://stacks.math.columbia.edu/tag/033H}{Section 033H}]{stacks-project} that every connected component of \(M_d(X_{[i,j]},X^{[n,1]},g.D^{[k]})\) is irreducible. So it suffices to show \\\(M_d(X_{[i,j]},X^{[n,1]},g.D^{[k]})\) is connected.

By \cite[Theorem 2]{kleiman:transversality}, \(M_d(X_{[i,j]},X^{[n,1]},g.D^{[k]})\) is equidimensional and every component meets \(M_{0,3}(X,d)\) and hence \(\mathring{M}_d(X_{[i,j]},X^{[n,1]},g.D^{[k]})\). Consider the morphism \[\rho_3^{g,k}:M_d(X_{[i,j]},X^{[n,1]},g.D^{[k]})\to M_{d_1}(p_1(X_{[i,j]}),p_1(X^{[n,1]}),g.p_1(D^{[k]})).\] By \Lemma{param}, the general fibre intersected with \(\mathring{M}_d(X_{[i,j]},X^{[n,1]},g.D^{[k]})\) is connected, so at most one component of \(M_d(X_{[i,j]},X^{[n,1]},g.D^{[k]})\) maps dominantly.
On the other hand, since \(\rho_3^{g,k}\) is closed and surjective by \Corollary{surj3} and \(M_{d_1}(p_1(X_{[i,j]}),p_1(X^{[n,1]}),g.p_1(D^{[k]}))\) is irreducible by \Lemma{rational}, at least one component maps surjectively to it. Therefore, exactly one component \(Y\) maps surjectively to \\\(M_{d_1}(p_1(X_{[i,j]}),p_1(X^{[n,1]}),g.p_1(D^{[k]}))\). 

Now consider \[\rho_2: M_d(X_{[i,j]},X^{[n,1]})\to M_{d_1}(p_1(X_{[i,j]}),p_1(X^{[n,1]})).\] By \Lemma{fibrecomp} and \Lemma{surj2}, every irreducible component of the general fibre meets \(\mathring{M}_d(X_{[i,j]},X^{[n,1]})\), and it follows from \Lemma{param} that the general fibre is irreducible. Let \(U\) be an open dense subset of \(M_{d_1}(p_1(X_{[i,j]}),p_1(X^{[n,1]}))\) over which fibres of \(\rho_2\) are irreducible and of expected dimension \[m\coloneqq\dim M_d(X_{[i,j]},X^{[n,1]})-\dim M_{d_1}(p_1(X_{[i,j]}),p_1(X^{[n,1]})).\]
By \cite[Theorem 2]{kleiman:transversality}, every component of \(M_d(X_{[i,j]},X^{[n,1]},g.D^{[k]})\) meets \(\rho_2^{-1}(U)\).

Suppose \(Z\) is another component of \(M_d(X_{[i,j]},X^{[n,1]},g.D^{[k]})\). Since \(Z\cap\rho_2^{-1}(U)\) is non-empty, \(\rho_2(Z)\cap U\) is also non-empty and therefore dense in \(\rho_2(Z)\). 

By generic flatness, the general fibre of \(Z\to\rho_2(Z)\) has dimension \[m'\coloneqq\dim M_d(X_{[i,j]},X^{[n,1]},g.D^{[k]})-\dim \rho_2(Z).\] Note that \[m-m'=1+\dim \rho_2(Z)-\dim M_{d_1}(p_1(X_{[i,j]}),p_1(X^{[n,1]})).\] Since \(M_d(X_{[i,j]},X^{[n,1]})\supseteq Z\), we must have \(m\geq m'\), or equivalently \begin{equation}\label{eqn:geq}\dim \rho_2(Z)\geq\dim M_{d_1}(p_1(X_{[i,j]}),p_1(X^{[n,1]}))-1.\end{equation} 

When \(k=1\), \eqn{geq} is equivalent to \begin{equation}\label{eqn:k=1}\dim \rho_2(Z)\geq\dim M_{d_1}(p_1(X_{[i,j]}),p_1(X^{[n,1]}),g.p_1(D_1)).\end{equation} \(M_{d_1}(p_1(X_{[i,j]}),p_1(X^{[n,1]}),g.p_1(D_1))\) is irreducible by \Lemma{rational}, so \eqn{k=1} implies \\\(\rho_2(Z)=M_{d_1}(p_1(X_{[i,j]}),p_1(X^{[n,1]}),g.p_1(D_1))\), which is a contradiction. 

When \(k=2\), we have \[M_{d_1}(p_1(X_{[i,j]}),p_1(X^{[n,1]}),g.p_1(D_2))=M_{d_1}(p_1(X_{[i,j]}),p_1(X^{[n,1]})).\] Since \(Z\) doesn't map surjectively, \eqn{geq} must be an equality, which is equivalent to \(m'=m\). This implies \(Z\) contains an irreducible fibre of \(\rho_2\), but then \(Z\) meets \(Y\), which is a contradiction. 

Therefore, \(M_d(X_{[i,j]},X^{[n,1]},g.D^{[k]})\) is connected for \(g\) general in \(G\).
\end{proof}

\begin{thm}\label{thm:ratcon}
The general fibre of 
\begin{equation*}\label{resev3}\ev_3:M_d(X_{[i,j]},D^{[k]})\to\Gamma_d(X_{[i,j]},D^{[k]})\end{equation*} is rationally connected.
\end{thm}

\begin{proof}
When \(d=0\), the domain and target are both isomorphic to \(X_{[i,j]}\cap D^{[k]}\) and \(\ev_3\) is an isomorphism. 

When \(d>0\), without loss of generality, we assume \eqn{assumption}.

\item[Case 1: $d_1=0$.]

Assume \(M_d(X_{[i,j]},D^{[k]})\) is non-empty. By \Lemma{fibrecomp} and \Lemma{reduced}, there is a dense open subset \(U\subseteq \Gamma_d(X_{[i,j]},D^{[k]})\) such that for all \(u\in U\), \(\ev_3^{-1}(u)\) is reduced and every irreducible component meets \(\mathring{M}_d(X_{[i,j]},D^{[k]})\). On the other hand, since \(\mathring{M}_0(p_1(X_{[i,j]}),p_1(D^{[k]}),p_1(u))\) is a point,  \(\ev_3^{-1}(u)\cap\mathring{M}_d(X_{[i,j]},D^{[k]})=\mathring{M}_d(X_{[i,j]},D^{[k]},u)\) is rational by \Lemma{param}. Therefore, \(\ev_3^{-1}(u)\) is rational.

\item[Case 2: $d_1\geq1$.]

In this case $\Gamma_d(X_{[i,j]},D^{[k]})=\Gamma_d(X_{[i,j]})=X$. So the general fibre considered is 
\\$M_d(X_{[i,j]},D^{[k]},g.X^{[n,1]})$ for general $g$. 
By \Corollary{surj3}, \Lemma{rational} and \Proposition{irred}, the restriction \[\rho: M_d(X_{[i,j]},D^{[k]},g.X^{[n,1]})\to M_{d_1}(p_1(X_{[i,j]}),p_1(D^{[k]}),g.p_1(X^{[n,1]}))\] is a surjective morphism between complex irreducible varieties and the target is rational. By \Lemma{fibrecomp} and \Lemma{reduced}, the general fibre is reduced and every irreducible component meets \(\mathring{M}_d(X_{[i,j]},D^{[k]},g.X^{[n,1]})\). It then follows from \Lemma{param} that the general fibre is rational. \cite[Corollary 1.3]{GHS} implies that \(M_d(X_{[i,j]},D^{[k]},g.X^{[n,1]})\) is rationally connected.
\end{proof}

\begin{example}\label{example:birat}
When \(d_1=d_2=1\), for a general point \(C=(C_1,C_2)\in X\),
 \(\ev_1: M_d(X_{[i,j]}, D^{[k]}, C)\to X_{[i,j]}\) is birational.
\end{example}
\begin{proof}
Without loss of generality, assume \(k=1\).

Let \(Y=\{(U,V)\in X:\ U\subset C_2\}\) and \(Z=\{(U,V)\in X:\ C_1\subset V\}\). Since \(C\) is general, we may assume \[X_{[i,j]}\nsubseteq Y\cup Z\text{ and }C_1\notin p_1(D^{[1]}).\] Let \((A_1,A_2)\) be a point in \(X_{[i,j]}\) such that \[A_1\notin p_1(D^{[1]}),\ A_1\nsubseteq C_2,\text{ and }C_1\nsubseteq A_2.\] We shall construct a unique preimage \(\mu: \P^1\to X\) in \(\mathring{M}_d(X_{[i,j]}, D^{[1]}, C)\).

Note that \(\C^n=A_1\oplus(A_2\cap C_2)\oplus C_1\). \(A_1\) and \(C_1\) determines a line \(L\) in \(\Gr(1,n)\). \(A_1,\ C_1\notin p_1(D^{[1]})\) implies \(L\nsubseteq p_1(D^{[1]})\), so \(L\) intersects \(p_1(D^{[1]})\) at one point \(B_1\). Let \(\{v\}\) be a basis for \(B_1\). Since \(B_1\subset A_1\oplus C_1\), we can write  \(v=v_{0}+v_{\infty}\), where \(v_{0}\in A_1, v_{\infty}\in C_1\) are unique and nonzero. Let \(\mu: \P^1\to X\) be given by \[(s:t)\mapsto (\text{Span}\{t v_0+s v_\infty\}, \text{Span}\{t v_0+s v_\infty\}\oplus (A_2\cap C_2)).\] Then \(\mu\) is the unique degree \((1,1)\) morphism \(\mu: \P^1\to X\) such that \[\mu(0)=(A_1,A_2),\ \mu(1)\in D^{[1]},\ \mu(\infty)=(C_1,C_2).\]
\end{proof}

\begin{cor}\label{cor:cohomtriv}
\(\ev_3: M_d(X_{[i,j]},D^{[k]})\to\Gamma_d(X_{[i,j]},D^{[k]})\) is cohomologically trivial.
\end{cor}

\begin{proof}
    This follows from \Theorem{ratcon} and \cite[Theorem 3.1]{qkgrass} (see also \cite[Proposition 5.2]{Buch2013}).
\end{proof}

\section{Chevalley Formula for \(T\)-equivariant Quantum \(K\)-theory}\label{sec:chev}

See \cite{chriss-ginzburg} for an introduction to equivariant \(K\)-theory and \cite[Section 2]{qkchev} for a summary of the (equivariant) (quantum) \(K\)-theory of flag varieties. We will use the same set-up and notations as in \cite{qkchev}.

\subsection{Pullbacks along Morphisms of Algebraic Groups}\label{sec:pullback}  
Let \(G, H\) be algebraic groups and \(f: G\to H\) a morphism. Let \(Z\) be an \(H\)-variety. Then \(Z\) is also a \(G\)-variety with the action given by \(g.z\coloneqq f(g).z\) for all \(g\in G,\ z\in Z\), and an \(H\)-equivariant \(\cO_Z\)-module \(\cF\) is naturally also a \(G\)-equivariant \(\cO_Z\)-module, which we denote by \(f^*\cF\). \([\cF]\mapsto [f^*\cF]\) defines a morphism \(f^*:K_H(Z)\to K_G(Z)\), where \(K_H(Z)\) is the Grothendieck group of \(H\)-equivariant coherent \(\cO_Z\)-modules. \(K_H(Z)\) is a module over the Grothendieck ring \(K^H(Z)\) of \(H\)-equivariant vector bundles over \(Z\), and \(f^*\) is compatible with the module structures in the sense that the diagram
\[\begin{CD}K^H(Z)\times K_H(Z)     @> >>  K_H(Z)\\@VVf^*\times f^*V @VVf^*V       \\ K^G(Z)\times K_G(Z) @> >> K_G(Z)\end{CD}\] commutes. 

\begin{lemma}\label{lemma:diags}
    Let \(G, H\) be algebraic groups and \(f: G\to H\) a morphism. 
    \begin{enumerate}
        \item Let \(\lambda:H\to GL(V)\) be a representation, and \([V_\lambda]\) the corresponding class in \(K_H(\pt)\). Then \[f^*([V_\lambda])=[V_{\lambda\circ f}].\]
        \item Let \(Z\) be an \(H\)-variety and \(\Omega\subseteq Z\) an \(H\)-stable closed subvariety, then \[f^*([\cO_\Omega])=[\cO_\Omega].\]
        \item Given \(\phi: Y\to Z\) a flat \(H\)-equivariant morphism of \(H\)-varieties, the diagram \[\begin{CD}K_H(Z)     @>f^*>>  K_G(Z)\\@VV\phi^*V @VV\phi^*V       \\ K_H(Y) @>f^*>> K_G(Y)\end{CD}\] commutes.
        \item Given \(\phi: Y\to Z\) a proper \(H\)-equivariant morphism of \(H\)-varieties, the diagram \[\begin{CD}K_H(Y)     @>f^*>>  K_G(Y)\\@VV\phi_*V @VV\phi_*V       \\ K_H(Z) @>f^*>> K_G(Z)\end{CD}\] commutes. 
    \end{enumerate}
\end{lemma}

\begin{proof}
    This is clear.
\end{proof}

\subsection{Symmetry of Gromov--Witten Invariants}\label{sec:symm}

By \Corollary{map}, composing with \(\varphi\) gives an isomorphism \[\Phi: M_d\to M_{\varphi_*d}.\] \(\varphi\) and \(\Phi\) are not \(T\)-equivariant with repect to the usual \(T\)-actions. However, they become equivariant if we modify the \(T\)-actions on the targets by composing with \(\tilde{\varphi}:T\to T\). Let \(X'\) and \(M_{\varphi_*d}'\) be \(X\) and \(M_{\varphi_*d}\) equipped with the modified \(T\)-actions, respectively, and \(\ev_i': M_{\varphi_*d}'\to X'\) be the evaluation maps. Let \[\varphi':X\to X'\text{ and }\Phi': M_d\to M_{\varphi_*d}'\] be \(\varphi\) and \(\Phi\) with the \(T\)-actions on the targets modified, respectively. 

\begin{defn}
    Define \(\varphi^*:K_T(X)\to K_T(X)\) by \[\varphi^*\coloneqq {\varphi'}^*{\tilde{\varphi}}^*,\] where \({\tilde{\varphi}}^*: K_T(X)\to K_T(X')\) is an example of the construction in \Section{pullback}.
\end{defn}

Recall that for \(d=(d_1,d_2)\in H_2(X)^+\), \(\varphi_*d=(d_2,d_1).\)

\begin{lemma}\label{lemma:symmI} For \(\sigma_1,\dots,\sigma_m\in K^T(X),\)
\[I_d^T(\varphi^*\sigma_1,\dots,\varphi^*\sigma_m)=\tilde\varphi^*I_{\varphi_*d}^T(\sigma_1,\dots,\sigma_m).\]
\end{lemma}
\begin{proof}
The evaluation maps are \(G\)-equivariant and \(G\) acts on \(X\) transitively. By generic flatness, the evaluation maps are flat. Since 
\[
    \begin{CD}
        X@>\varphi'>> X'\\@AA{\ev_i}A @AA{ev'_i}A\\M_d@>\Phi'>>M'_{\varphi_*d}\\@V{pr}VV @V{pr'}VV\\\pt @>{id}>> \pt
    \end{CD}
\] is a commutative diagram of proper, flat, \(T\)-equivariant morphisms,
\begin{align*}
    I_d^T(\varphi^*\sigma_1,\dots,
    \varphi^*\sigma_m)&={pr}_*({\ev_1}^*{\varphi'}^*{\tilde{\varphi}}^*\sigma_1\cdots{\ev_m}^*{\varphi'}^*{\tilde{\varphi}}^*\sigma_m)
    \\&={pr}_*{\Phi'}^*({ev'_1}^*{\tilde{\varphi}}^*\sigma_1\cdots {ev'_m}^*{\tilde{\varphi}}^*\sigma_m)
    \\&={pr}'_*({ev'_1}^*{\tilde{\varphi}}^*\sigma_1\cdots {ev'_m}^*{\tilde{\varphi}}^*\sigma_m).\label{eqn:pr_2_*}
\end{align*}
By \Lemma{diags}, this is then equal to \(\tilde{\varphi}^*I_{\varphi_*d}^T(\sigma_1,\dots,\sigma_m)\).
\end{proof}

\subsection{Chevalley Formula for \(T\)-equivariant Quantum \(K\)-theory}

\begin{notation}\label{notation:tilde}
    From now on, when we write ``\(\equiv\)'' or ``\(\not\equiv\)'', it shall be understood that the equivalence is taken modulo \(n\).

    Let \[\widetilde{W^P}\coloneqq\{[i,j]\in\Z\times\Z: i\not\equiv j\}.\]  For \(w=[i,j]\in\widetilde{W^P}\), we define \[\cO^{w}\coloneqq q^{d(w)}[\cO_{X^{\overline w}}]\in {QK_T(X)}_q,\] where \(\overline{w}\coloneqq[\overline{i},\overline{j}]\in W^P\) is defined by \(\overline{i}\equiv i\) and \(\overline{j}\equiv j\), \[d(w)\coloneqq(\frac{i-\overline{i}}{n},\frac{\overline{j}-j}{n}),\ q^{(d_1,d_2)}\coloneqq {q_1}^{d_1}{q_2}^{d_2}\text{ for }d_1,d_2\in\Z,\] and the subscript \(q\) stands for localization with respect to \(\{q^d:\ d\in \Z^2_{\geq0}\}.\) Recall that \(\varepsilon_i: T\to \C^*\) is the character that sends a diagonal matrix to its \(i\)-th diagonal entry. We let \[\varepsilon_i\coloneqq\varepsilon_{\overline{i}}\text{ for }i\in\Z,\]
    and write \(\C_\eta\) for the one-dimensional \(T\)-representation corresponding to the character \(\eta\), and \([\C_\eta]\) for the corresponding class in \(K^T(\pt)\). 
    
    Recall that \(K^T(X)\) is a free module over \(K^T(\pt)\). The classes \(\cO^{w},\ w\in W^P\) form a basis. For \(w\in W^P\), let \(\cO_w^\vee\in K^T(X)\) be the basis element dual to \(\cO^w\) in the sense that \(\chi^T_X(\cO^u\cdot\cO_v^\vee)=\delta_{uv}\) for \(u,\ v\in W^P\), where \(\chi_X\) is the sheaf Euler characteristic map. Let \(\mathcal I_{\partial X_w}\subseteq \cO_{X_w}\) be the ideal sheaf of \(\partial X_w\coloneqq X_w\setminus BwP\). Then \(\cO_w^\vee=[\mathcal{I}_{\partial X_w}]\) (see \cite{Brion}). We extend the definition to allow for \(w\in \widetilde{W^P}\) by letting \[\cO^\vee_{w}\coloneqq q^{d(w)}\cO^\vee_{\overline w}\in {QK_T(X)}_q.\] 

    For \(i=1,\ 2\), we will write \(\cO^{[i]}\) for the class of \(\cO_{D^{[i]}}\) in \(K^T(X)\).
\end{notation}

The goal of this subsection is to prove the following \(T\)-equivariant Chevalley formula and derive some immediate consequences.

\begin{thm}\label{thm:qkTchev} In \({QK_T(X)}_q\), for \([i,j]\in \widetilde{W^P}\), the quantum multiplication by divisor classes is given by
    \begin{equation}\label{eqn:qkchev1}\cO^{[i,j]}\star\cO^{[1]}=\begin{cases}(1-[\C_{\varepsilon_i-\varepsilon_1}])\cO^{[i,j]}+[\C_{\varepsilon_i-\varepsilon_1}]\cO^{[i+1,j]}&\text{if }i+1\not\equiv j\\ (1-[\C_{\varepsilon_i-\varepsilon_1}])\cO^{[i,j]}+[\C_{\varepsilon_i-\varepsilon_1}](\cO^{[i+1,j-1]}+\cO^{[i+2,j]}-\cO^{[i+2,j-1]})&\text{if }i+1\equiv j\end{cases},\end{equation}
    \begin{equation}\label{eqn:qkchev2}\cO^{[i,j]}\star\cO^{[2]}=\begin{cases}(1-[\C_{\varepsilon_n-\varepsilon_j}])\cO^{[i,j]}+[\C_{\varepsilon_n-\varepsilon_j}]\cO^{[i,j-1]}&\text{if }i+1\not\equiv j\\ (1-[\C_{\varepsilon_n-\varepsilon_j}])\cO^{[i,j]}+[\C_{\varepsilon_n-\varepsilon_j}](\cO^{[i+1,j-1]}+\cO^{[i,j-2]}-\cO^{[i+1,j-2]})&\text{if }i+1\equiv j\end{cases}.\end{equation}
\end{thm}

The Chevalley formula for the equivariant ordinary \(K\)-theory of \(X\) can be recovered from \Theorem{qkTchev} by interpreting \(\cO^w\) as \(0\) unless \(w\in W^P\). We rely on \cite{LP} for this special case. The formula is derived from \cite[Corollary 7.1 and Corollary 8.2]{LP} using that \((\alpha_{1,2},\dots,\alpha_{1,n})\) is a reduced \(\varepsilon_1\)-chain in the sense of \cite[Definition 5.4]{LP}.

Note that \Corollary{cohomtriv} 
gives us the following ``quantum-equals-classical" type result, which allows us to compute some equivariant \(K\)-theoretic Gromov--Witten invariants using the Chevalley formula for equivariant ordinary \(K\)-theory.

\begin{cor}\label{cor:qclassical} For \(\sigma\in K^T(X)\), \([i,j]\in W^P\), and \(k=1,\ 2\),
\[I_d^T(\sigma,\cO^{[k]},\cO_{[i,j]})=\begin{cases}\chi_X^T(\sigma\cdot[\cO_{\Gamma_d(X_{[i,j]})}])\text{ if }d_k>0\\ \chi_X^T(\sigma\cdot\cO^{[k]}\cdot[\cO_{\Gamma_d(X_{[i,j]})}])\text{ if }d_k=0
\end{cases}.\]
\end{cor}

\begin{proof}
\begin{equation*}
    \begin{split}
        I_d^T(\sigma,\cO^{[k]},\cO_{[i,j]})&=\chi_{M_d}^T(\ev_1^*(\sigma)\cdot \ev_2^*(\cO^{[k]})\cdot \ev_3^*(\cO_{[i,j]}))
        \\&=\chi_{M_d}^T(\ev_1^*(\sigma)\cdot[\cO_{M_d(X_{[i,j]},D^{[k]})}])
        \\&=\chi_X^T(\sigma\cdot{\ev_1}_*([\cO_{M_d(X_{[i,j]},D^{[k]})}]))\label{eqn:cohomtriv}
        \\&=\chi_X^T(\sigma\cdot[\cO_{\Gamma_d(X_{[i,j]},D^{[k]})}])
    \end{split}.
\end{equation*}
For the second equality, see \cite[Section 4.1]{qkgrass} and the references therein. The third equality follows from the projection formula. The last equality follows from \Corollary{cohomtriv}. 

The rest follows from \Lemma{gamma}.
\end{proof}

\begin{notation}
    For \(\sigma_1,\ \sigma_2\in K^T(X)\) and \(w\in\widetilde{W^P}\), we write \(I^T(\sigma_1,\sigma_2;\cO_{w})\) for \(I_{d(w)}^T(\sigma_1,\sigma_2,\cO_{\overline w})\) and \(I^T(\sigma_1,\sigma_2;\cO_{w}^\vee)\) for \(I_{d(w)}^T(\sigma_1,\sigma_2,\cO_{\overline{w}}^\vee)\), with the convention that they are \(0\) when \(d(w)\not\geq 0\). We will only allow the third argument to be outside \(K^T(X)\), and the semicolon is there to remind the reader of it. We do the same in the non-equivariant case.
\end{notation}

When specialized to the non-equivariant case, the following proposition says that for general \(g\) and \(k=1,\ 2\), \begin{equation}\label{eqn:01}\chi(\cO_{M_d(g.X^u,D^{[k]},X_v)})=1\text{ whenever }M_d(g.X^u,D^{[k]},X_v)\neq\varnothing.\end{equation} In other words, when \(M_d(g.X^u,D^{[k]},X_v)\) is non-empty, its arithmetic genus is \(0\). See \Corollary{genus0} for a generalization.
Gromov--Witten varieties associated to Schubert varieties in general position are known to have rational singularities \cite[Corollary 3.1]{Buch2013}. By \cite[Corollary 4.18(a)]{debarre2013higher}, \eqn{01} will follow if \(M_d(g.X^u,D^{[k]},X_v)\) is rationally connected whenever non-empty. We have proved this when \(X^u\) is a point (see the proof of \Theorem{ratcon}).

\begin{prop}\label{prop:GW} 
For \([i,j]\in W^P,\ [k,l]\in\widetilde{W^P}\),
\begin{equation}\label{eqn:I1}I^T(\cO^{[i,j]},\cO^{[1]};\cO_{[k,l]})=\begin{cases}1-[\C_{\varepsilon_k-\varepsilon_1}]&\text{if }i=k\text{ and } j\geq l\\1&\text{if }i<k\text{ and } j\geq l\\0&\text{otherwise }\end{cases};\end{equation}
\begin{equation}\label{eqn:I2}I^T(\cO^{[i,j]},\cO^{[2]};\cO_{[k,l]})=\begin{cases}1-[\C_{\varepsilon_n-\varepsilon_l}]&\text{if }i\leq k\text{ and }j=l \\1&\text{if }i\leq k\text{ and }j>l \\0&\text{otherwise }\end{cases}.\end{equation}
\end{prop}

\begin{proof} 
Assume \(i<k\) and \(j\geq l\). 

If \(k\leq n\), then \[\overline k=k>i,\] and by \Corollary{qclassical} \begin{equation}\label{eqn:I}I^T(\cO^{[i,j]},\cO^{[1]};\cO_{[k,l]})=\chi_X^T(\cO^{[i,j]}\cdot\cO^{[1]}\cdot[\cO_{\Gamma_{(0,d_2)}(X_{[k,\overline{l}]})}]),\end{equation} where \(d_2=(\overline{l}-l)/{n}\). Let \[X_{[a,b]}={\Gamma_{(0,d_2)}(X_{[k,\overline{l}]})}.\] By \Section{curvenbhd}, \begin{equation}\label{eqn:ab}n\geq a=k>i\text{ and }b\leq j.\end{equation} By the Chevalley formula for the equivariant ordinary \(K\)-theory of \(X\), \[\cO^{[i,j]}\cdot\cO^{[1]}=(1-[\C_{\varepsilon_i-\varepsilon_1}])\cO^{[i,j]}+[\C_{\varepsilon_i-\varepsilon_1}]S,\] where \[S=\begin{cases}\cO^{[i+1,j]}&\text{if }i+1\neq j\\\cO^{[i+1,j-1]}&\text{if }i+1=j=n\\\cO^{[i+1,j-1]}+\cO^{[i+2,j]}-\cO^{[i+2,j-1]}&\text{if }i+1=j<n\end{cases}.\] The conditions \eqn{ab} implies \[a=n>b\text{ when }i+1=j=n\] and \[\text{either }a>i+1\text{ or }b<j\text{ when }i+1=j<n.\]
Since for \([x,y],\ [a,b]\in W^P\), \begin{equation}\label{eqn:chi}
    \chi_X^T(\cO^{[x,y]}\cdot\cO_{[a,b]})=
    \begin{cases}
        1&\text{if }a\geq x\text{ and }b\leq y\\0&\text{if }a<x\text{ or }b>y
    \end{cases},
\end{equation} we have \[\chi_X^T(S\cdot\cO_{[a,b]})=1\] and the right hand side of \eqn{I} equals \((1-[\C_{\varepsilon_i-\varepsilon_1}])+[\C_{\varepsilon_i-\varepsilon_1}]=1\).

If \(k>n\), then by \Corollary{qclassical} \begin{equation}\label{eqn:I>}I^T(\cO^{[i,j]},\cO^{[1]};\cO_{[k,l]})=\chi_X^T(\cO^{[i,j]}\cdot[\cO_{\Gamma_{d}(X_{[\overline k,\overline l]})}]),\end{equation} where \[d=(\frac{k-\overline{k}}{n},\frac{\overline{l}-l}{n}).\] Let \[X_{[a,b]}=\Gamma_d(X_{[\overline{k},\overline{l}]}).\] If \(l<1\), then \(X_{[a,b]}=X\) and the right hand side of \eqn{I>} equals \(1\). Otherwise
\[\overline l=l\leq j.\] By \Section{curvenbhd}, \[a\geq i\text{ and }b=l\leq j.\] By \eqn{chi}, the right-hand side of \eqn{I>} equals \(1\).

The other two cases of \eqn{I1} are similar, and the proof is omitted. \eqn{I2} follows from \eqn{I1} by applying \Lemma{symmI} and \Lemma{diags}.
\end{proof}

\begin{defn}\label{defn:circPsi}
    Following \cite{qkchev}, for \(\sigma_1,\ \sigma_2\in K^T(X)\), we define the formal sum \[\sigma_1\odot\sigma_2=\sum_{w\in W^P,\ d\geq0}I^T_d(\sigma_1,\sigma_2,\cO_w^\vee)q^d\cO^w=\sum_{w\in\widetilde{W^P}}I^T(\sigma_1,\sigma_2;\cO_w^\vee)\cO^w,\] and we define the automorphism \(\Psi: QK_T(X)_q\to QK_T(X)_q\) of \(K^T(\pt)(\!(q_1,q_2)\!)\)-modules by \[\Psi(\cO^{u})=q^{d(u)}\Psi(\cO^{\overline{u}})=\sum_{d\geq0}q^{d(u)+d}[\cO_{\Gamma_d(X^{\overline u})}]\text{ for }u\in \widetilde{W^P}.\] We will write \(\Gamma_d(X^u)\) for \(\Gamma_{d-d(u)}(X^{\overline{u}})\) when \(d\geq d(u)\) and set \(\Gamma_d(X^u)=\varnothing\) when \(d\not\geq d(u)\). Then we have for \(u\in \widetilde{W^P}\), \begin{equation}\label{eqn:Psi}\Psi(\cO^{u})=\sum_{v\in\widetilde{W^P}:\ \Gamma_{d(v)}(X^{u})=X^{\overline v}}\cO^{v}.\end{equation}
\end{defn}

The following was proved in \cite[Proposition 2.3]{qkchev}. We include a proof for completeness.

\begin{prop}[Buch--Chaput--Mihalcea--Perrin] \label{prop:circ}
    For \(\sigma_1,\ \sigma_2\in K^T(X)\), their product \(\sigma_1\star\sigma_2\) in \(QK_T(X)\) satisfies \(\sigma_1\odot\sigma_2=\Psi(\sigma_1\star\sigma_2).\)
\end{prop}
\begin{proof}
    By linearity, it suffices to prove that for \(u,v\in W^P\),
    \begin{equation}\label{eqn:circ}
        \cO^u\odot\cO^v=\Psi(\cO^u\star\cO^v).
    \end{equation}
    Write \(\cO^u\star\cO^v=\sum_{z\in\widetilde{W^P}}N_{u,v}^z\cO^z\), then 
    \[
        \Psi(\cO^u\star\cO^v)=\sum_{z\in\widetilde{W^P}}N_{u,v}^z\sum_{w\in\widetilde{W^P}:\  \Gamma_{d(w)}(X^z)=X^{\overline{w}}}\cO^w.
    \]
    It suffices to show that \(\cO^w\) has the same coefficients on both sides of \eqn{circ}, which means
    \begin{equation}\label{eqn:I^T_vee}
        I^T(\cO^u,\cO^v;\cO_w^\vee)=\sum_{z\in\widetilde{W^P}:\ \Gamma_{d(w)}(X^z)=X^{\overline{w}}}N_{u,v}^z.
    \end{equation}
    By definition,  
    \[
        (\!(\cO^u\star\cO^v,\cO_{\overline{w}}^\vee)\!)=\sum_{d}q^d I_d^T(\cO^u,\cO^v,\cO_{\overline{w}}^\vee),
    \] 
    where \((\!(\sigma_1,\sigma_2)\!)=\sum_d q^d I_d^T(\sigma_1,\sigma_2)\) for \(\sigma_1,\sigma_2\in K^T(X)\). Therefore, 
    \begin{equation*}
        \sum_{z\in\widetilde{W^P}}N_{u,v}^z\sum_e q^{e+d(z)}I_e^T(\cO^{\overline{z}},\cO_{\overline{w}}^\vee)=\sum_d q^d I_d^T(\cO^u,\cO^v,\cO_{\overline{w}}^\vee).
    \end{equation*}
    Taking out the coefficients of \(q^{d(w)}\), we have 
    \begin{equation*}
        \sum_{z\in\widetilde{W^P}}N_{u,v}^z\sum_{e:\ e+d(z)=d(w)}I_e^T(\cO^{\overline{z}},\cO_{\overline{w}}^\vee)=I^T(\cO^u,\cO^v;\cO_w^\vee).
    \end{equation*}
    By \cite{Buch2013}*{Proposition 3.2}, 
    \[
        I_e^T(\cO^{\overline{z}},\cO_{\overline{w}}^\vee)=\chi_X^T([\cO_{\Gamma_e(X^{\overline{z}})}],\cO_{\overline w}^\vee)=
        \begin{cases}
            1 & \Gamma_e(X^{\overline z})=X^{\overline w}\\
            0 & \text{otherwise}
        \end{cases}.
    \]
    Equation \eqn{I^T_vee} follows.
\end{proof}

Since \(\sigma_1\star\sigma_2=\Psi^{-1}(\sigma_1\odot\sigma_2)\) for \(\sigma_1,\ \sigma_2\in K^T(X)\) by \Proposition{circ}, the verification of \Theorem{qkTchev} requires computing \(I^T(\cO^u,\cO^{[i]};\cO_w^\vee)\) for all \(u\in W^P, w\in\widetilde{W^P}\); \(i=1,\ 2\). In \Proposition{GW}, we have computed \(I^T(\cO^u,\cO^{[i]};\cO_w)\). We now express each \(\cO_w^\vee\) as a linear combination of the classes \(\cO_w\).

Recall from \Section{def} that we identify \(W^P\) with the set of minimal coset representatives of \(W/W_P\). 
\begin{defn}
    For \(v\in W^P\), we define \(I(v)\) to be the set of \(u\in W^P\) such that every element in \(W\) between \(u\) and \(v\) is contained in \(W^P\), i.e., \[I(v)\coloneqq\{u\in W^P:\ u\leq v\text{ and }\nexists\ w\in W\setminus W^P\text{ s.t. }u<w<v\}.\]
\end{defn}

The next lemma is a special case of \cite[Theorem 1.2 and Section 3]{Deodhar}.

\begin{lemma}
Let \(f\) be any function on \(W^P\) with values in an abelian group and define function \(g\) on \(W^P\) by \[g(u)=\sum_{v\in W^P:\ v\leq u}f(v).\] Then for all \(v\in W^P\), \[f(v)=\sum_{u\in I(v)}(-1)^{\ell(u)+\ell(v)}g(u).\]
\end{lemma}

\begin{lemma}
    For \(u\in W^P\), \[\cO_{u}=\sum_{v\in W^P:\ v\leq u}\cO_{v}^\vee.\]
\end{lemma}

\begin{proof}
    Since \(\cO_v^\vee\), \(v\in W^P\) form a basis, we can formally write \(\cO_u=\sum_{v\in W^P}c_v\cO_v^\vee\) with \(c_v\in K^T(\pt)\) and hence \(\chi_X^T(\cO^w\cdot\cO_u)=\sum_{v\in W^P}c_v\chi_X^T(\cO^w\cdot\cO_v^\vee)\). Recall that \(\chi^T_X(\cO^w\cdot\cO_v^\vee)=\delta_{wv}\). On the other hand,
    \[
    \chi_X^T(\cO^w\cdot\cO_u)=\chi_X^T([\cO_{X^w\cap X_u}])=
    \begin{cases}
        1 & w\leq u\\
        0 & w\not\leq u
    \end{cases},
\] 
using that the Richardson variety \(X^w\cap X_u\) is non-empty if and only if \(w\leq u\), and in this case it is rational with rational singularities. The result follows.
\end{proof}

\begin{cor}\label{cor:expand} For \(v\in W^P\),
\[\cO_{v}^\vee=\sum_{u\in I(v)}(-1)^{\ell(u)+\ell(v)}\cO_{u}.\]
\end{cor}

Now we give explicit computation of \(I(v)\). Some examples are illustrated in \Figure{I}.

\begin{lemma}\label{lemma:necc}
    Let \(u,\ v\in W^P\). Assume \(u\in I(v)\), then \(u'\in I(v)\) for all \(u'\in W^P\) such that \(u\leq u'\leq v\).
\end{lemma}
\begin{proof}
    This is clear.
\end{proof}

\begin{lemma}\label{lemma:I} 
    For \([a,b]\in W^P\), \[I([a,b])=
    \begin{cases}
        W^P\cap\{[a,b],\ [a-1,b],\ [a,b+1],\ [a-1,b+1]\}&\text{if }a-b\neq1,\ 2\\
        W^P\cap\{[a,b],\ [b-1,b],\ [b,a],\ [a,a+1],\ [b-1,a],\ [b,a+1]\}&\text{if }a-b=1\\
        \{[k,l]\in W^P:\ b\leq k, l\leq a\}&\text{if }a-b=2
    \end{cases}.\]
\end{lemma}
\begin{proof}
    By reasoning as follows, the conclusion can be read from the poset \(W^P\). When \(n=5\), the poset structure of \(W^P\) is illustrated in \Figure{W^P}. 

    Let \(v=[a,b]\). Let \(u\in W^P\) be such that \(u\leq v\). 
    
    If \(l(u)\geq l(v)-1\), then \(u\in I(v)\).

    If \(l(u)=l(v)-2\), then the Bruhat interval \(\bm [ u,v\bm ]\) in \(W\) contains \(4\) elements. \(u\in I(v)\) if and only if \(\bm [ u,v\bm ]\cap W^P\) contains \(4\) elements.

    Now assume \(l(u)\leq l(v)-3\). 
    
    If \(a-b\neq 1,\ 2\), then neither \([a-2,b]\) nor \([a,b+2]\) is in \(I(v)\), and that at least one of them is in \(W^P\cap\bm [ u,v\bm]\). By \Lemma{necc}, \(u\not\in I(v)\). 
    
    If \(a-b=1\) or \(2\), then neither \([a-3,b]\) nor \([a,b+3]\) is in \(I(v)\), and that at least one of them is in \(W^P\cap\bm [ u,v\bm]\) except when we are in either of the following cases: 
    \begin{enumerate}     
        \item \(a-b=1,\ u=[b-1,a+1].\)
        \item \(a-b=2,\ u=[b,a];\)
    \end{enumerate}
    By \Lemma{necc}, \(u\not\in I(v)\) unless we are in one of those two cases. The Bruhat interval \(\bm [ u,v\bm ]\) in \(W\) contains \(8\) elements in case (1) and \(6\) elements in case (2). \(W^P\cap\bm [ u,v\bm]\) contains \(7\) elements in case (1) and \(6\) elements in case (2). We conclude that \(u\not\in I(v)\) in case (1) and \(u\in I(v)\)  in case (2).
\end{proof}

\begin{center}
    \begin{figure}[ht]
        \begin{tikzpicture}
            \node (51) at (0,0) {\([5,1]\)};
            \node [below left of=51] (41) {\([4,1]\)};
            \node [below right of=51] (52) {\([5,2]\)};
            \node [below left of=41] (31)  {\([3,1]\)};
            \node [below right of=41] (42) {\([4,2]\)};
            \node [below right of=52] (53) {\([5,3]\)};
            \node [below left of=31] (21) {\([2,1]\)};
            \node [below right of=31] (32) {\([3,2]\)};
            \node [below right of=42] (43) {\([4,3]\)};
            \node [below right of=53] (54) {\([5,4]\)};
            \node [below of=21] (12) {\([1,2]\)};
            \node [below of=32] (23) {\([2,3]\)};
            \node [below of=43] (34) {\([3,4]\)};
            \node [below of=54] (45) {\([4,5]\)};
            \node [below right of=12] (13) {\([1,3]\)};
            \node [below right of=23] (24) {\([2,4]\)};
            \node [below left of=45] (35) {\([3,5]\)};
            \node [below right of=13] (14) {\([1,4]\)};
            \node [below left of=35] (25) {\([2,5]\)};
            \node [below left of=25] (15) {\([1,5]\)};
            \draw [shorten <=-3pt, shorten >=-3pt] (51) -- (41);
            \draw [shorten <=-3pt, shorten >=-3pt] (51) -- (52);
            \draw [shorten <=-3pt, shorten >=-3pt] (41) -- (31);
            \draw [shorten <=-3pt, shorten >=-3pt] (41) -- (42);
            \draw [shorten <=-3pt, shorten >=-3pt] (52) -- (42);
            \draw [shorten <=-3pt, shorten >=-3pt] (52) -- (53);
            \draw [shorten <=-3pt, shorten >=-3pt] (31) -- (21);
            \draw [shorten <=-3pt, shorten >=-3pt] (31) -- (32);
            \draw [shorten <=-3pt, shorten >=-3pt] (42) -- (32);
            \draw [shorten <=-3pt, shorten >=-3pt] (42) -- (43);
            \draw [shorten <=-3pt, shorten >=-3pt] (53) -- (43);
            \draw [shorten <=-3pt, shorten >=-3pt] (53) -- (54);
            \draw [shorten <=-3pt, shorten >=-3pt] (21) -- (12);
            \draw [shorten <=-3pt, shorten >=-3pt] (21) -- (23);
            \draw [shorten <=-3pt, shorten >=-3pt] (32) -- (12);
            \draw [shorten <=-3pt, shorten >=-3pt] (32) -- (23);
            \draw [shorten <=-3pt, shorten >=-3pt] (43) -- (23);
            \draw [shorten <=-3pt, shorten >=-3pt] (43) -- (45);
            \draw [shorten <=-3pt, shorten >=-3pt] (32) -- (34);
            \draw [shorten <=-3pt, shorten >=-3pt] (54) -- (34);
            \draw [shorten <=-3pt, shorten >=-3pt] (54) -- (45);
            \draw [shorten <=-3pt, shorten >=-3pt] (43) -- (34);
            \draw [shorten <=-3pt, shorten >=-3pt] (12) -- (13);
            \draw [shorten <=-3pt, shorten >=-3pt] (23) -- (13);
            \draw [shorten <=-3pt, shorten >=-3pt] (23) -- (24);
            \draw [shorten <=-3pt, shorten >=-3pt] (34) -- (24);
            \draw [shorten <=-3pt, shorten >=-3pt] (34) -- (35);
            \draw [shorten <=-3pt, shorten >=-3pt] (45) -- (35);
            \draw [shorten <=-3pt, shorten >=-3pt] (13) -- (14);
            \draw [shorten <=-3pt, shorten >=-3pt] (24) -- (14);
            \draw [shorten <=-3pt, shorten >=-3pt] (24) -- (25);
            \draw [shorten <=-3pt, shorten >=-3pt] (35) -- (25);
            \draw [shorten <=-3pt, shorten >=-3pt] (14) -- (15);
            \draw [shorten <=-3pt, shorten >=-3pt] (25) -- (15);
        \end{tikzpicture}
        \caption{Hasse diagram of $(W^P, \leq)$ when \(n=5\).}
        \label{fig:W^P}
    \end{figure}

    \begin{figure}[ht]
        \begin{subfigure}{0.25\textwidth}
            \centering
            \begin{tikzpicture}
                \node (41) at (0,0) {\([4,1]\)};
                \node [below left of=41] (31)  {\([3,1]\)};
                \node [below right of=41] (42) {\([4,2]\)};
                \node [below right of=31] (32) {\([3,2]\)};
                \draw [shorten <=-3pt, shorten >=-3pt] (41) -- (31);
                \draw [shorten <=-3pt, shorten >=-3pt] (41) -- (42);
                \draw [shorten <=-3pt, shorten >=-3pt] (31) -- (32);
                \draw [shorten <=-3pt, shorten >=-3pt] (42) -- (32);
            \end{tikzpicture}
            \caption{\(I([4,1])\).}
        \end{subfigure}
        \begin{subfigure}{0.25\textwidth}
            \centering
            \begin{tikzpicture}
                \node (32) at (0,0) {\([3,2]\)};
                \node [below of=32] (23) {\([2,3]\)};
                \node [left of=23] (12) {\([1,2]\)};
                \node [right of=23] (34) {\([3,4]\)};
                \node [below of=12] (13) {\([1,3]\)};
                \node [below of=34] (24) {\([2,4]\)};
                \draw [shorten <=-3pt, shorten >=-3pt] (32) -- (12);
                \draw [shorten <=-3pt, shorten >=-3pt] (32) -- (23);
                \draw [shorten <=-3pt, shorten >=-3pt] (32) -- (34);
                \draw [shorten <=-3pt, shorten >=-3pt] (12) -- (13);
                \draw [shorten <=-3pt, shorten >=-3pt] (23) -- (13);
                \draw [shorten <=-3pt, shorten >=-3pt] (23) -- (24);
                \draw [shorten <=-3pt, shorten >=-3pt] (34) -- (24);
            \end{tikzpicture}
            \caption{\(I([3,2])\).}
        \end{subfigure}
        \begin{subfigure}{0.25\textwidth}
            \centering
            \begin{tikzpicture}
                \node (31) at (0,0) {\([3,1]\)};
                \node [below left of=31] (21) {\([2,1]\)};
                \node [below right of=31] (32) {\([3,2]\)};
                \node [below of=21] (12) {\([1,2]\)};
                \node [below of=32] (23) {\([2,3]\)};
                \node [below right of=12] (13) {\([1,3]\)};
                \draw [shorten <=-3pt, shorten >=-3pt] (31) -- (21);
                \draw [shorten <=-3pt, shorten >=-3pt] (31) -- (32);
                \draw [shorten <=-3pt, shorten >=-3pt] (21) -- (12);
                \draw [shorten <=-3pt, shorten >=-3pt] (21) -- (23);
                \draw [shorten <=-3pt, shorten >=-3pt] (32) -- (12);
                \draw [shorten <=-3pt, shorten >=-3pt] (32) -- (23);
                \draw [shorten <=-3pt, shorten >=-3pt] (12) -- (13);
                \draw [shorten <=-3pt, shorten >=-3pt] (23) -- (13);
            \end{tikzpicture}
            \caption{\(I([3,1])\).}
        \end{subfigure}
        \begin{subfigure}{0.2\textwidth}
            \centering
            \begin{tikzpicture}
                \node (12) at (0,0) {\([1,2]\)};
                \node [below of=12] (13) {\([1,3]\)};
                \draw [shorten <=-3pt, shorten >=-3pt] (12) -- (13);
            \end{tikzpicture}
            \caption{\(I([1,2])\).}
         \end{subfigure}
         \caption{Hasse diagrams of $(I(v), \leq)$ for various values of $v$.}
         \label{fig:I}
    \end{figure}   
\end{center}

\begin{defn}\label{defn:length}
    For \(u=[x,y]\in \widetilde{W^P}\), we now let 
    \begin{align*}
        \ell([x,y])&=\text{codim}X^{\overline{u}}+\int_{d(u)}c_1(T_X)\\&=\ell(\overline{u})+\frac{n-1}{n}(x-\overline x-y+\overline y)\\&=(n-1)(1+\frac{x-y}{n})+\frac{\overline x-\overline y}{n}-\chi(\overline x>\overline y),
    \end{align*}
    where \(T_X\) is the tangent bundle of \(X\), and \(\int_{d(u)}c_1(T_X)\) is the degree of \(q^{d(u)}\) in the (graded) quantum cohomology ring of \(X\).

    For \(u,\ v\in \widetilde{W^P}\), we write \(u\in I(v)\) if \[d(u)=d(v)\text{ and }\overline u\in I(\overline v).\] 
\end{defn}

\begin{proof}[Proof of \Theorem{qkTchev}]
    Without loss of generality, we may assume \([i,j]\in W^P\). We will prove 
    \begin{equation}\label{eqn:odot1}\cO^{[i,j]}\odot\cO^{[1]}=\begin{cases}(1-[\C_{\varepsilon_i-\varepsilon_1}])\Psi(\cO^{[i,j]})+[\C_{\varepsilon_i-\varepsilon_1}]\Psi(\cO^{[i+1,j]})&\text{if }i+1\not\equiv j\\(1-[\C_{\varepsilon_i-\varepsilon_1}])\Psi(\cO^{[i,j]})+[\C_{\varepsilon_i-\varepsilon_1}]\Psi(\cO^{[i+1,j-1]}+\cO^{[i+2,j]}-\cO^{[i+2,j-1]})&\text{if }i+1\equiv j\end{cases}.\end{equation}
    By \cite[Proposition 2.3]{qkchev}, this is equivalent to \eqn{qkchev1}. We will prove \eqn{odot1} by comparing the coefficient of \(\cO^{w}\) for \(w=[a,b]\in\widetilde{W^P}\) on both sides. Let \(d=d(w)=(d_1,d_2)\).

    In order to compute the coefficient of \(\cO^w\) on the right hand side of \eqn{odot1}, we list curve neighborhoods in \Table{<>} through \Table{n,1} (see \Section{curvenbhd}).

    \begin{table}[h!]
        \caption{Assume \(i<n\) and \(i+1\neq j\).}
        \label{tab:<>}
        \begin{center}
            \begin{tabular}{|c|c|c|}
                \hline
                & \(\Gamma_d(X^{[i,j]})\) & \(\Gamma_d(X^{[i+1,j]})\)\\ \hline
                \(d=0\) & \(X^{[i,j]}\) & \(X^{[i+1,j]}\) \\\hline
                \(d_1>0\) and \(d_2=0\) & \(X^{[1,j]}\) or \(X^{[2,1]}\) & \(X^{[1,j]}\) or \(X^{[2,1]}\) \\\hline
                \(d_1=0\) and \(d_2>0\) & \(X^{[i,n]}\) & \(X^{[i+1,n]}\) or \(X^{[n,n-1]}\)\\\hline
                \(d_1,\ d_2>0\) & \(X^{[1,n]}\) & \(X^{[1,n]}\)\\\hline
            \end{tabular}
        \end{center}
    \end{table}

    \begin{table}[h!]
        \caption{Assume \(i=n\) and \(j>1\).}
        \label{tab:=>}
        \begin{center}
            \begin{tabular}{|c|c|c|}
                \hline
                & \(\Gamma_d(X^{[n,j]})\) & \(\Gamma_d(X^{[n+1,j]})\)\\ \hline
                \(d=0\) & \(X^{[n,j]}\) & \(\varnothing\) \\\hline
                \(d_1>0\) and \(d_2=0\) & \(X^{[1,j]}\) & \(X^{[1,j]}\)\\\hline
                \(d_1=0\) and \(d_2>0\) & \(X^{[n,n-1]}\) & \(\varnothing\)\\\hline
                \(d_1,\ d_2>0\) & \(X^{[1,n]}\) & \(X^{[1,n]}\)\\\hline
            \end{tabular}
        \end{center}
    \end{table}

    When \(i+1\not\equiv j\), using the definition of \(\Psi\) (equation \eqn{Psi}), \Table{<>}, and \Table{=>}, we see that the coefficient of \(\cO^w\) on the right-hand side of \eqn{odot1} equals
    \begin{enumerate}
        \item[\namedlabel{itm:a}{(a)}] \(1-[\C_{\varepsilon_i-\varepsilon_1}]\) when we are in one of the following situations:
        \begin{enumerate}
            \item[\namedlabel{itm:a1}{(a1)}] \([a,b]=[i,j]\); 
            \item[\namedlabel{itm:a2}{(a2)}] \(a=i< n\) and \(b=n-d_2n\) for some \(d_2>0\);
            \item[\namedlabel{itm:a3}{(a3)}] \(a=i=n\) and \(b=n-1-d_2n\) for some \(d_2>0\);
        \end{enumerate}
        \item[\namedlabel{itm:b}{(b)}] \([\C_{\varepsilon_i-\varepsilon_1}]\) when we are in one of the following situations:
        \begin{enumerate}
            \item[\namedlabel{itm:b1}{(b1)}] \(a=i+1\leq n\) and \(b=j\); 
            \item[\namedlabel{itm:b2}{(b2)}] \(a=i+1<n\) and \(b=n-d_2n\) for some \(d_2>0\);
            \item[\namedlabel{itm:b3}{(b3)}] \(a=i+1=n\) and \(b=n-1-d_2n\) for some \(d_2>0\);
        \end{enumerate} 
        \item[\namedlabel{itm:c}{(c)}] \(1\) when we are in one of the following situations:
        \begin{enumerate}
            \item[\namedlabel{itm:c1}{(c1)}]  \(a=1+d_1 n\) for some \(d_1>0\) and \(b=j>1\);
            \item[\namedlabel{itm:c2}{(c2)}]  \(a=2+d_1 n\) for some \(d_1>0\) and \(b=j=1\);
            \item[\namedlabel{itm:c3}{(c3)}]  \([a,b]=[1+d_1 n,n-d_2 n]\) for some \(d_1,\ d_2>0\);
        \end{enumerate}
        \item[\namedlabel{itm:d}{(d)}]  \(0\) in all other situations.
    \end{enumerate}

    \begin{table}[h!]
        \caption{Assume \(i+1=j<n\).}
        \label{tab:<}
        \begin{center}
            \begin{tabular}{|c|c|c|c|c|}
                \hline
                & \(\Gamma_d(X^{[i,j]})\) & \(\Gamma_d(X^{[i+1,j-1]})\) & \(\Gamma_d(X^{[i+2,j]})\) & \(\Gamma_d(X^{[i+2,j-1]})\)\\ \hline
                \(d=0\) & \(X^{[i,j]}\) & \(X^{[i+1,j-1]}\) & \(X^{[i+2,j]}\) & \(X^{[i+2,j-1]}\) \\\hline
                \(d_1>0\) & \multirow{2}{3em}{\centering\(X^{[1,j]}\)} & \(X^{[1,j-1]}\) & \multirow{2}{4em}{\centering\(X^{[1,j]}\)} & \(X^{[1,j-1]}\) \\
                \(d_2=0\) & & or \(X^{[2,1]}\) & & or \(X^{[2,1]}\)\\\hline
                \(d_1=0\) & \multirow{2}{3em}{\centering\(X^{[i,n]}\)} & \multirow{2}{4em}{\centering\(X^{[i+1,n]}\)} & \(X^{[i+2,n]}\)  &\(X^{[i+2,n]}\) \\
                \(d_2>0\) & & & or \(X^{[n,n-1]}\) & or \(X^{[n,n-1]}\)\\ \hline
                \(d_1,\ d_2>0\) & \(X^{[1,n]}\) & \(X^{[1,n]}\) & \(X^{[1,n]}\) & \(X^{[1,n]}\) \\\hline
            \end{tabular}
        \end{center}
    \end{table}

    \begin{table}[h!]
        \caption{Assume \([i,j]=[n-1,n]\).}
        \label{tab:n-1,n}
        \begin{center}
            \begin{tabular}{|c|c|c|c|c|}
                \hline
                & \(\Gamma_d(X^{[n-1,n]})\) & \(\Gamma_d(X^{[n,n-1]})\) & \(\Gamma_d(X^{[n+1,n]})\) & \(\Gamma_d(X^{[n+1,n-1]})\)\\ \hline
                \(d=0\) & \(X^{[n-1,n]}\) & \(X^{[n,n-1]}\) & \(\varnothing\) & \(\varnothing\) \\\hline
                \(d_1>0\) and \(d_2=0\) & \(X^{[1,n]}\) & \(X^{[1,n-1]}\) & \(X^{[1,n]}\) & \(X^{[1,n-1]}\)\\\hline
                \(d_1=0\) and \(d_2>0\) & \(X^{[n-1,n]}\) & \(X^{[n,n-1]}\) & \(\varnothing\) &\(\varnothing\)\\\hline
                \(d_1,\ d_2>0\) & \(X^{[1,n]}\) & \(X^{[1,n]}\) & \(X^{[1,n]}\) & \(X^{[1,n]}\) \\\hline
            \end{tabular}
        \end{center}
    \end{table}

    \begin{table}[h!]
        \caption{Assume \([i,j]=[n,1]\).}
        \label{tab:n,1}
        \begin{center}
            \begin{tabular}{|c|c|c|c|c|}
                \hline
                & \(\Gamma_d(X^{[n,1]})\) & \(\Gamma_d(X^{[n+1,0]})\) & \(\Gamma_d(X^{[n+2,1]})\) & \(\Gamma_d(X^{[n+2,0]})\)\\ \hline
                \(d=0\) & \(X^{[n,1]}\) & \(\varnothing\) & \(\varnothing\) & \(\varnothing\) \\\hline
                \(d_1>0\) and \(d_2=0\) & \(X^{[2,1]}\) & \(\varnothing\) & \(X^{[2,1]}\) & \(\varnothing\)\\\hline
                \(d_1=0\) and \(d_2>0\) & \(X^{[n,n-1]}\) & \(\varnothing\) & \(\varnothing\) &\(\varnothing\)\\\hline
                \(d_1=1,\ d_2>0\) & \(X^{[1,n]}\) & \(X^{[1,n]}\) & \(X^{[2,n]}\) & \(X^{[2,n]}\) \\\hline
                \(d_1>1,\ d_2>0\) & \(X^{[1,n]}\) & \(X^{[1,n]}\) & \(X^{[1,n]}\) & \(X^{[1,n]}\)\\\hline
            \end{tabular}
        \end{center}
    \end{table}

    When \(i+1\equiv j\), using the definition of \(\Psi\) (equation \eqn{Psi}) and \Table{<} through \Table{n,1}, we see that the coefficient of \(\cO^w\) on the right-hand side of \eqn{odot1} equals
    \begin{enumerate}
        \item[\namedlabel{itm:e}{(e)}] \(1-[\C_{\varepsilon_i-\varepsilon_1}]\) when we are in one of the following situations:
        \begin{enumerate}
            \item[\namedlabel{itm:e1}{(e1)}] \([a,b]=[i,j]\); 
            \item[\namedlabel{itm:e2}{(e2)}] \(a=i< n\) and \(b=n-d_2n\) for some \(d_2>0\);
            \item[\namedlabel{itm:e3}{(e3)}] \(a=i=n\) and \(b=n-1-d_2n\) for some \(d_2>0\);
        \end{enumerate}
        \item[\namedlabel{itm:f}{(f)}] \([\C_{\varepsilon_i-\varepsilon_1}]\) when we are in one of the following situations:
        \begin{enumerate}
            \item[\namedlabel{itm:f1}{(f1)}] \(a=i+1=j\) and \(b=j-1=i\);
            \item[\namedlabel{itm:f2}{(f2)}] \(a=i+2\leq n\) and \(b=j\); 
            \item[\namedlabel{itm:f3}{(f3)}] \(a=i+1<n\) and \(b=n-d_2n\) for some \(d_2>0\);
            \item[\namedlabel{itm:f4}{(f4)}] \(a=i+1=n\) and \(b=n-1-d_2n\) for some \(d_2>0\);
        \end{enumerate} 
        \item[\namedlabel{itm:g}{(g)}] \(-[\C_{\varepsilon_i-\varepsilon_1}]\) when \(a=i+2\leq n\) and \(b=j-1=i\);
        \item[\namedlabel{itm:h}{(h)}] \(1\) when we are in one of the following situations:
        \begin{enumerate}
            \item[\namedlabel{itm:h1}{(h1)}]  \(a=1+d_1 n\) for some \(d_1>0\) and \(b=j>1\);
            \item[\namedlabel{itm:h2}{(h2)}]  \(a=2+d_1 n\) for some \(d_1>0\) and \(b=j=1\);
            \item[\namedlabel{itm:h3}{(h3)}]  \([a,b]=[1+d_1 n,n-d_2 n]\) for some \(d_1,\ d_2>0\);
        \end{enumerate}
        \item[\namedlabel{itm:i}{(i)}]  \(0\) in all other situations.
    \end{enumerate}

The coefficient of \(\cO^w\) on the left-hand side of \eqn{odot1} is \(I^T(\cO^{[i,j]},\cO^{[1]};\cO_{[a,b]}^\vee)\), which we will now compute and compare with \ref{itm:a} through \ref{itm:i}. By \Corollary{expand}, \[I^T(\cO^{[i,j]},\cO^{[1]};\cO_{[a,b]}^\vee)=\sum_{[k,l]\in I([a,b])}(-1)^{\ell([k,l])+\ell([a,b])}I^T(\cO^{[i,j]},\cO^{[1]};\cO_{[k,l]}).\] By \Proposition{GW}, 
\begin{equation*}
    I^T(\cO^{[i,j]},\cO^{[1]};\cO_{[k,l]})=
    \begin{cases}
        1-[\C_{\varepsilon_k-\varepsilon_1}]&\text{if }k=i,\ l\leq j\\1&\text{if }k>i,\ l\leq j\\0&\text{otherwise }
    \end{cases}.
\end{equation*} 
Note that \(I^T(\cO^{[i,j]},\cO^{[1]};\cO_{[k,l]})\) is nonzero only when \(l\leq j\), and in this case, the value of \(I^T(\cO^{[i,j]},\cO^{[1]};\cO_{[k,l]})\) depends solely on \(k\). Let \[S_k=\sum_{l:\ [k,l]\in I([a,b]),\ l\leq j} (-1)^{\ell([k,l])+\ell([a,b])}, \] then 
\begin{equation}\label{eqn:I^T}
    I^T(\cO^{[i,j]},\cO^{[1]};\cO_{[a,b]}^\vee)=(1-[\C_{\varepsilon_i-\varepsilon_1}])S_{i} +\sum_{k>i}S_k.
\end{equation}
\(S_k\) is easily computable using \Lemma{I}. For example, when \(1<a-1=b+1=j<n\), \[S_{a}=(-1)^{\ell([a,b])+\ell([a,b])}+(-1)^{\ell([a,b+1])+\ell([a,b])}=0.\] For another example, when \(a\not\equiv 1\) and \(b=0\), \[S_{a-1}=(-1)^{\ell([a-1,b])+\ell([a,b])}=-1\] (note that \([a-1,b+1]\not\in I([a,b])\) because \(d([a-1,b+1])\neq d([a,b])\)). We shall consider the three cases in \Lemma{I} separately. \Table{A}, \Table{B}, and \Table{C} list values of \(S_k\) in each case, respectively. When the conditions in the tables are not satisfied, \(S_k=0\).

\begin{table}[h!]
    \caption{Values of \(S_k\) when \(\overline a-\overline b\neq 1,\ 2\).}
    \label{tab:A}
    \begin{center}
        \begin{tabular}{|c|c|}
            \hline
            & either \(b=j\) or \(b=n-d_2 n\) for some \(d_2>0\)\\ \hline
            \(k=a\) & \(1\)\\\hline
            \(k=a-1\) and \(a\not\equiv1\) & \(-1\)\\\hline
        \end{tabular}
    \end{center}
\end{table}

\item[Case 1: Assume \(\overline a-\overline b\neq1,\ 2\).] \Table{A} says that out of all \(S_k\),  only \(S_a\) and \(S_{a-1}\) can be nonzero when \(a\not\equiv 1\), and only \(S_a\) can be nonzero when \(a\equiv 1\). Therefore, we can reduce \eqn{I^T} to a finite sum. 
    
When \(a\not\equiv1\), equation \eqn{I^T} becomes 
\begin{align*}
    I^T(\cO^{[i,j]},\cO^{[1]};\cO_{[a,b]}^\vee)=&
    \begin{cases}
        0&\text{if }i>a\\
        (1-[\C_{\varepsilon_i-\varepsilon_1}])S_a&\text{if }i=a\\
        (1-[\C_{\varepsilon_i-\varepsilon_1}])S_{a-1}+S_{a}&\text{if }i=a-1\\
        S_{a-1}+S_a&\text{if }i<a-1
    \end{cases}.
\end{align*}
By \Table{A}, \(S_a\) and \(S_{a-1}\) are nonzero only when either \(b=j\) or \(b=n-d_2 n\) for some \(d_2>0\), and in either case they are \(1\) and \(-1\), respectively. Therefore, \(I^T(\cO^{[i,j]},\cO^{[1]};\cO_{[a,b]}^\vee)\) equals
\begin{enumerate}
    \item \(1-[\C_{\varepsilon_i-\varepsilon_1}]\) when \([a,b]=[i,j]\);
    \item \(1-[\C_{\varepsilon_i-\varepsilon_1}]\) when \(a=i<n\) and \(b=n-d_2n\) for some \(d_2>0\);
    \item \([\C_{\varepsilon_i-\varepsilon_1}]\) when \(a=i+1\leq n\) and \(b=j\);
    \item \([\C_{\varepsilon_i-\varepsilon_1}]\) when \(a=i+1< n\) and \(b=n-d_2 n\) for some \(d_2>0\);
    \item \(0\) otherwise.
\end{enumerate}
When \(\overline a-\overline b\neq1,\ 2\) and \(a\not\equiv1\), the situations \ref{itm:a1}, \ref{itm:a2}, \ref{itm:b1}, \ref{itm:b2}, \ref{itm:d}, \ref{itm:e1}, \ref{itm:e2}, \ref{itm:f3}, and \ref{itm:i} are realizable while the other ones are not. In all realizable situations, the value of \(I^T(\cO^{[i,j]},\cO^{[1]};\cO_{[a,b]}^\vee)\) agrees with the coefficient of \(\cO^w\) on the right hand side of \eqn{odot1}.

By \Table{A}, when \(a\equiv1\), equation \eqn{I^T} becomes 
\begin{align*}
    I^T(\cO^{[i,j]},\cO^{[1]};\cO_{[a,b]}^\vee)=&
    \begin{cases}
        0&\text{if }i>a\\
        (1-[\C_{\varepsilon_i-\varepsilon_1}])S_a&\text{if }i=a=1\\
        S_{a}&\text{if }i<a
    \end{cases}.
\end{align*}   
Substituting values for \(S_a\), we get that \(I^T(\cO^{[i,j]},\cO^{[1]};\cO_{[a,b]}^\vee)\) equals 
\begin{enumerate}
    \item \(1-[\C_{\varepsilon_i-\varepsilon_1}]=0\) when \(a=i=1\) and either \(b=j>1\) or \(b=n-d_2n\) for some \(d_2>0\);
    \item \(1\) when \(a=1+d_1n\) for some \(d_1>0\) and either \(b=j>1\) or \(b=n-d_2n\) for some \(d_2>0\);
    \item \(0\) otherwise. 
\end{enumerate} 
When \(a\equiv1\), the situations \ref{itm:a1}, \ref{itm:a2}, \ref{itm:c1}, \ref{itm:c3}, \ref{itm:d}, \ref{itm:e1}, \ref{itm:e2}, \ref{itm:h1}, \ref{itm:h3}, and \ref{itm:i} are realizable while the other ones are not. In all realizable situations, the value of \(I^T(\cO^{[i,j]},\cO^{[1]};\cO_{[a,b]}^\vee)\) agrees with the coefficient of \(\cO^w\) on the right-hand side of \eqn{odot1}.

\begin{table}[h!]
    \caption{Values of \(S_k\) when \(\overline a-\overline b=1\).}
    \label{tab:B}
    \begin{center}
        \begin{tabular}{|c|c|c|}
            \hline
            & \multirow{2}{3em}{\centering\(b=j\)}  & either \(b=j-1\geq 1\) or \\ & & \(b=n-1-d_2n\) for some \(d_2>0\) \\ \hline
            \(k=a\) & \(1\) & \(1\)\\\hline
            \(k=a-1\) & \(0\) & \(-1\) \\
            \hline
            \(k=a-2\) and \(a\not\equiv 2\) & \(-1\) & \(0\) \\ \hline
        \end{tabular}
    \end{center}
\end{table}        

\item[Case 2: Assume \(\overline a-\overline b=1\).]

By \Table{B}, when \(a\not\equiv 2\), equation \eqn{I^T} becomes 
\begin{align*}
    I^T(\cO^{[i,j]},\cO^{[1]};\cO_{[a,b]}^\vee)=&
    \begin{cases}
        0&\text{if }i>a\\(1-[\C_{\varepsilon_i-\varepsilon_1}])S_a&\text{if }i=a\\(1-[\C_{\varepsilon_i-\varepsilon_1}])S_{a-1}+S_{a}&\text{if }i=a-1\\(1-[\C_{\varepsilon_i-\varepsilon_1}])S_{a-2}+S_{a-1}+S_a&\text{if }i=a-2\\S_{a-2}+S_{a-1}+S_a&\text{if }i<a-2
    \end{cases}.
\end{align*}     
Substituting values for \(S_k\), we get that \(I^T(\cO^{[i,j]},\cO^{[1]};\cO_{[a,b]}^\vee)\) equals 
\begin{enumerate}
    \item \(1-[\C_{\varepsilon_i-\varepsilon_1}]\) when \([a,b]=[i,j]\) and \(i-j=1\);
    \item \(1-[\C_{\varepsilon_i-\varepsilon_1}]\) when \(a=i=n\) and \(b=n-1-d_2n\) for some \(d_2>0\);
    \item \([\C_{\varepsilon_i-\varepsilon_1}]\) when 
    \(a=i+1=j\) and \(b=j-1=i\);
    \item \([\C_{\varepsilon_i-\varepsilon_1}]\) when 
    \(a=i+1=n\) and \(b=n-1-d_2n\) for some \(d_2>0\);
    \item \([\C_{\varepsilon_i-\varepsilon_1}]\) when \(a=i+2\leq n\) and \(b=i+1=j\);
    \item \(0\) otherwise.
\end{enumerate}
When \(\overline{a}-\overline{b}=1\) and \(a\not\equiv 2\), the situations \ref{itm:a1}, \ref{itm:a3}, \ref{itm:b3}, \ref{itm:d}, \ref{itm:e3}, \ref{itm:f1}, \ref{itm:f2}, \ref{itm:f4}, and \ref{itm:i} are realizable while the other ones are not. In all realizable situations, the value of \(I^T(\cO^{[i,j]},\cO^{[1]};\cO_{[a,b]}^\vee)\) agrees with the coefficient of \(\cO^w\) on the right-hand side of \eqn{odot1}. 

When \(a\equiv2\), we have \(b\equiv 1\). By \Table{B}, equation \eqn{I^T} becomes 
\begin{align*}
    I^T(\cO^{[i,j]},\cO^{[1]};\cO_{[a,b]}^\vee)=&
    \begin{cases}
        0&\text{if }i>a\\(1-[\C_{\varepsilon_i-\varepsilon_1}])S_a&\text{if }i=a=2\\(1-[\C_{\varepsilon_i-\varepsilon_1}])S_{a-1}+S_{a}&\text{if }i=a-1=1\\S_{a-1}+S_a&\text{if }i<a-1
    \end{cases}\\=&
    \begin{cases}
        1-[\C_{\varepsilon_i-\varepsilon_1}]&\text{if }[i,j]=[a,b]=[2,1]\\
        [\C_{\varepsilon_i-\varepsilon_1}]=1&\text{if }[i,j]=[1,2]\text{ and }[a,b]=[2,1]\\1&\text{if }a=2+d_1 n\text{ for some }d_1>0\text{ and }b=j=1\\
        0&\text{otherwise }
    \end{cases}.
\end{align*}
When \(a\equiv 2\) and \(b\equiv 1\), the situations \ref{itm:a1}, \ref{itm:c2}, \ref{itm:d}, \ref{itm:f1}, \ref{itm:h2}, and \ref{itm:i} are realizable while the other ones are not. In all realizable situations, the value of \(I^T(\cO^{[i,j]},\cO^{[1]};\cO_{[a,b]}^\vee)\) agrees with the coefficient of \(\cO^w\) on the right-hand side of \eqn{odot1}.     

\begin{table}[h!]
    \caption{Values of \(S_k\) when \(\overline a-\overline b=2\).}
    \label{tab:C}
    \begin{center}
        \begin{tabular}{|c|c|c|}
        \hline
            & \(b=j\)  & \(b=j-1\geq 1\)\\ \hline
            \(k=a\) & \(1\) & \(0\)  \\\hline
            \(k=a-1\) & \(-1\) & \(-1\) \\\hline
            \(k=a-2\) & \(0\) & \(1\) \\\hline
        \end{tabular}
    \end{center}
\end{table}        

\item[Case 3: Assume \(\overline a-\overline b= 2\).]
    
By \Table{C}, equation \eqn{I^T} becomes 
\begin{align*}
    I^T(\cO^{[i,j]},\cO^{[1]};\cO_{[a,b]}^\vee)=&
    \begin{cases}
        0&\text{if }i>a\\(1-[\C_{\varepsilon_i-\varepsilon_1}])S_a&\text{if }i=a\\(1-[\C_{\varepsilon_i-\varepsilon_1}])S_{a-1}+S_{a}&\text{if }i=a-1\\(1-[\C_{\varepsilon_i-\varepsilon_1}])S_{a-2}+S_{a-1}+S_a&\text{if }i=a-2\\S_{a-2}+S_{a-1}+S_a&\text{if }i<a-2
    \end{cases}\\=&
    \begin{cases}
        1-[\C_{\varepsilon_i-\varepsilon_1}]&\text{if }[a,b]=[i,j]\text{ and }i-j=2\\
        [\C_{\varepsilon_i-\varepsilon_1}]&\text{if }
        a=i+1\leq n\text{ and }b=j=i-1\\
        -[\C_{\varepsilon_i-\varepsilon_1}]&\text{if }a=i+2\leq n\text{ and }b=j-1=i\\0&\text{otherwise }
    \end{cases}.
\end{align*} 
In this case, the situations \ref{itm:a1}, \ref{itm:b1}, \ref{itm:d}, \ref{itm:e1}, \ref{itm:g}, and \ref{itm:i} are realizable while the other ones are not. In all realizable situations, the value of \(I^T(\cO^{[i,j]},\cO^{[1]};\cO_{[a,b]}^\vee)\) agrees with the coefficient of \(\cO^w\) on the right-hand side of \eqn{odot1}.

Equations \eqn{odot1} and therefore \eqn{qkchev1} follows. Applying \Lemma{symmI} and \Lemma{diags} produces \eqn{qkchev2}.
\end{proof}

As a consequence, we have the following algorithm that recursively expresses \(\cO^{[i,j]}\) as a polynomial in \(\cO^{[1]},\ \cO^{[2]},\ q_1,\ q_2\) with coefficients in \(K^T(\pt)\) for all \([i,j]\in {W^P}\).
\begin{algo}
\item[Case 1: Assume \(i<j=n\) or \(i>j+1\).] 
\[\cO^{1,n}=1;\ \cO^{[i,j]}=[\C_{\varepsilon_1-\varepsilon_{i-1}}](\cO^{[1]}-1)\star\cO^{[i-1,j]}+\cO^{[i-1,j]}\text{ when }i>1.\]
\item[Case 2: Assume \(i<j<n\).] \[\cO^{[i,j]}=[\C_{\varepsilon_{j+1}-\varepsilon_n}](\cO^{[2]}-1)\star\cO^{[i,j+1]}+\cO^{[i,j+1]}.\]
\item[Case 3: Assume \(i=j+1\).] \[\cO^{[i,j]}=[\C_{\varepsilon_i-\varepsilon_n}](\cO^{[2]}+[\C_{\varepsilon_n-\varepsilon_i}]-1)\star\cO^{[j,i]}-\cO^{[j,j-1]}+\cO^{[i,j-1]}.\]
\end{algo}

There is a conjecture (see \cite{qkchev} for example) about the positivity of structure constants in \(QK_T(G/P)\). In the case of incidence varieties, \cite[Conjecture 2.2]{qkchev} specializes to the following statement. 
\begin{conj}\label{conj:pos}
For \(u,\ v,\ w\in\widetilde{W^P}\), let \(N_{u,v}^{w}\) be the coefficient of \(\cO^{w}\) in the product \(\cO^{u}\star\cO^{v}\) in \({QK_T(X)}_q\), then \[(-1)^{\ell(u)+\ell(v)+\ell(w)}N_{u,v}^{w}\in\N[[\C_{\varepsilon_{r+1}-\varepsilon_r}]-1: r=1,\dots,n-1].\]
\end{conj}

It follows from \Theorem{qkTchev} that the conjecture holds when \(\cO^{{v}}\) is (a power of \(q\) times) the class of a Schubert divisor. Equivalently:

\begin{cor}\label{cor:chevpos}
For \(u,\ w\in\widetilde{W^P}\), \(k=1,\ 2\), let \(N_{u,k}^{w}\) be the coefficient of \(\cO^{w}\) in the product \(\cO^{u}\star\cO^{[k]}\) in \({QK_T(X)}_q\), then \[(-1)^{\ell(u)+1+\ell(w)}N_{u,k}^{w}\in\N[[\C_{\varepsilon_{r+1}-\varepsilon_r}]-1: r=1,\dots,n-1].\]
\end{cor}
\begin{proof}
    Without loss of generality, we may assume \(u=[i,j]\in W^P\). One checks that when \(i+1\not\equiv j\), 
    \begin{equation*}\label{eqn:1}
        \ell([i+1,j])=\ell([i,j-1])=\ell([i,j])+1;    
    \end{equation*} 
    when \(i+1\equiv j\), 
    \begin{equation*}\label{eqn:2}
        \ell([i+1,j-1])=\ell([i+2,j])=\ell([i,j-2])=\ell([i,j])+1,    
    \end{equation*} 
    \begin{equation*}\label{eqn:3}
        \ell([i+2,j-1])=\ell([i+1,j-2])=\ell([i,j])+2.
    \end{equation*} 
    Since one can write \([\C_{\varepsilon_i-\varepsilon_1}]-1,\ [\C_{\varepsilon_i-\varepsilon_1}],\ [\C_{\varepsilon_n-\varepsilon_j}]-1,\ [\C_{\varepsilon_n-\varepsilon_j}]\) as polynomials with positive coefficients in the classes \([\C_{\varepsilon_{r+1}-\varepsilon_r}]-1\), \eqn{qkchev1} and \eqn{qkchev2} are positive. 
\end{proof}

Let \(f: \{id\}\to T\) be the trivial group homomorphism. By \Section{pullback}, \(f^*: K^T(X)\to K(X)\) has the property that for \(\sigma_1,\dots, \sigma_m\in K^T(X)\), \[I_d(f^*\sigma_1,\dots,f^*\sigma_m)=f^*I_d^T(\sigma_1,\dots,\sigma_m).\] This implies that the Chevalley formula for non-equivariant quantum \(K\)-theory of \(X\) can be obtained by replacing all the \([\C_{\varepsilon_{r+1}-\varepsilon_r}]\) with \(1\) in \Theorem{qkTchev}.

\begin{cor}\label{cor:qkchev}In \({QK(X)}_q\), for \([i,j]\in\widetilde{W^P}\),
\[\cO^{[i,j]}\star \cO^{[1]}=\begin{cases}\cO^{[i+1,j]}&\text{if }i+1\not\equiv j\\\cO^{[i+1,j-1]}+\cO^{[i+2,j]}-\cO^{[i+2,j-1]}&\text{if }i+1\equiv j\end{cases},\]\[\cO^{[i,j]}\star \cO^{[2]}=\begin{cases}\cO^{[i,j-1]}&\text{if }i+1\not\equiv j\\ \cO^{[i+1,j-1]}+\cO^{[i,j-2]}-\cO^{[i+1,j-2]}&\text{if }i+1\equiv j\end{cases}.\]\end{cor}

Using \cite[Theorem 2.18]{kato} with the projection \(p_1: X\to\P^{n-1}\), we recover the following Chevalley formula for the equivariant quantum \(K\)-theory of projective spaces, which is a special case of \cite[Theorem 3.9]{qkchev}, \cite[Theorem I]{kouno2020}, and \cite[Theorem 8]{kouno2021}. 

\begin{notation}\label{notation:proj}
    Let \[\cO^k=\begin{cases}[\cO_{V_+(x_1,\dots,x_k)}]\in K^T(\P^{n-1})&\text{if }0\leq k<n\\q\cO^{k-n}&\text{if }n\leq k<2n\end{cases},\] where \(x_1,\dots,x_n\) are the homogeneous coordinates of \(\P^{n-1}\).
\end{notation}

\begin{cor}\label{cor:chevP^n}
    In \(QK_T(\P^{n-1})\), for \(k=1,\dots, n\), \[\cO^{k-1}\star\cO^1=
        (1-[\C_{\varepsilon_{k}-\varepsilon_1}])\cO^{k-1}+[\C_{\varepsilon_{k}-\varepsilon_1}]\cO^{k}.\]
\end{cor}

\section{Quantum \(K\)-theoretic Littlewood--Richardson Rule}\label{sec:LR}

The goal of this section is to prove the following closed formula for quantum \(K\)-theoretic Littlewood--Richardson coefficients as well as some corollaries.

\begin{thm}\label{thm:LR}In \(QK(X)\), for \([i,j],\ [k,l]\in W^P\),
\begin{equation}\label{eqn:LR} 
    \cO^{[i,j]}\star\cO^{[k,l]}=\begin{cases}\cO^{[x,y]}&\text{if }x-y<n[\chi(i>j)+\chi(k>l)]\\ \cO^{[x,y-1]}+\cO^{[x+1,y]}-\cO^{[x+1,y-1]}&\text{if }x-y\geq n[\chi(i>j)+\chi(k>l)]\end{cases},
\end{equation}where $x=i+k-1$, $y=j+l-n$.\end{thm}

\begin{remark}
    When \(i=j+1\), we have 
    \[
        \cO^{[j+1,j]}\star\cO^{[k,l]}=\cO^{[j+k,j+l-n]}\text{ and }\ell([j+1,j])+\ell([k,l])=\ell([j+k,j+l-n]),
    \] 
    where \(\ell([j+k,j+l-n])\) is computed by \Definition{length}. These are exactly the Seidel products and the same formula holds in quantum cohomology, see \cites{Bel04,CMP09,Seidel}.
\end{remark}

\begin{example}
    Some products in \(QK(\Fl(1,4;5))\) are as follows: 
    \[\begin{aligned}\cO^{[1,3]}\star\cO^{[1,5]}&=\cO^{[1,3]};\\
    \cO^{[1,2]}\star\cO^{[2,1]}&=\cO^{[2,-2]}=q_2\cO^{[2,3]};\\
    \cO^{[2,1]}\star\cO^{[5,1]}&=\cO^{[6,-3]}=q_1 q_2 \cO^{[1,2]};\\
    \cO^{[1,2]}\star\cO^{[3,5]}&=\cO^{[3,1]}+\cO^{[4,2]}-\cO^{[4,1]};\\
    \cO^{[2,3]}\star\cO^{[4,5]}&=\cO^{[5,2]}+\cO^{[6,3]}-\cO^{[6,2]}=\cO^{[5,2]}+q_1\cO^{[1,3]}-q_1\cO^{[1,2]};\\
    \cO^{[1,2]}\star\cO^{[1,2]}&=\cO^{[1,-2]}+\cO^{[2,-1]}-\cO^{[2,-2]}=q_2\cO^{[1,3]}+q_2\cO^{[2,4]}-q_2\cO^{[2,3]};\\
    \cO^{[1,2]}\star\cO^{[5,1]}&=\cO^{[5,-3]}+\cO^{[6,-2]}-\cO^{[6,-3]}=q_2\cO^{[5,2]}+q_1 q_2 \cO^{[1,3]}-q_1 q_2 \cO^{[1,2]};\\
    \cO^{[3,1]}\star\cO^{[5,1]}&=\cO^{[7,-4]}+\cO^{[8,-3]}-\cO^{[8,-4]}=q_1 q_2\cO^{[2,1]}+q_1 q_2 \cO^{[3,2]}-q_1 q_2 \cO^{[3,1]}.\end{aligned}\]
\end{example}

We will use \cite[Lemma 2.1]{buch2016puzzle}, which is restated below.

\begin{lemma}\label{lemma:alg}
Let $R$ be an associative ring with unit $1$.
Let $S\subset R$ be a subset that generates R as a $\mathbb{Z}$-algebra. Let $M$ be a left $R$-module.
Let $\mu : R\times M \to M$ be any $\mathbb{Z}$-bilinear map.
Assume that for all $r \in R$, $s \in S$, and $m \in M$ we have:
\begin{enumerate}
    \item $\mu(1,m)=m$;
    \item $\mu(rs,m)=\mu(r,sm)$.
\end{enumerate}
Then $\mu(r,m) = rm$ for all $r \in R$ and $m \in M$.
\end{lemma}

We will also use the following lemma.

\begin{lemma}\label{lemma:chi}
    Let \([a,b]\in W^P\). When \(a+1\not\equiv b\), we have: 
    \begin{equation}\label{eqn:<<}
        -n<a-b-n\chi(a>b)<-1;
    \end{equation}
    \begin{equation}
        \overline{a+1}-n\chi(\overline{a+1}>b)=a+1-n\chi(a>b).\label{eqn:i+1j}
    \end{equation}
    When \(a+1\equiv b\), we have: 
    \begin{equation}\label{eqn:=}
        a+1-b=n\chi(a>b);
    \end{equation}
    \begin{equation}\label{eqn:i+1j-1}
        \overline{a+1}-\overline{b-1}+n-1=n\chi(\overline{a+1}>\overline{b-1});
    \end{equation}
    \begin{equation}\label{eqn:i+2j}
        \overline{a+2}-b+n-1=n\chi(\overline{a+2}>b);
    \end{equation}
    \begin{align}\label{eqn:i+2j-1}
        \overline{a+2}-\overline{b-1}+n-2=n\chi(\overline{a+2}>\overline{b-1}).
    \end{align}
\end{lemma}
\begin{proof}
    \[a-b-n\chi(a>b)=\begin{cases}
        a-b&\text{if }a<b\\a-b-n&\text{if }a>b
    \end{cases}.\] This implies \eqn{<<} and \eqn{=}. 
    
    Assume \(a+1\not\equiv b\). \[\overline{a+1}-n\chi(\overline{a+1}>b)=
    \begin{cases}
        a+1-n\chi(a+1>b)&\text{if }a<n\\
        1&\text{if }a=n    
    \end{cases}.\] In either case, it is equal to \(a+1-n\chi(a>b)\). This proves \eqn{i+1j}. 
    
    Assume \(a+1\equiv b\). \[\overline{a+1}-\overline{b-1}+n-1=\begin{cases}
        a+1-b+n=n&\text{if }a+1=b\\0&\text{if }[a,b]=[n,1]
    \end{cases}.\] In either case, it is equal to \(n\chi(\overline{a+1}>\overline{b-1})\). This proves \eqn{i+1j-1}. \[\overline{a+2}-b+n-1=
    \begin{cases}
        a+1-b+n=n&\text{if }a+1=b<n\\n&\text{if }[a,b]=[n,1]\\0&\text{if }[a,b]=[n-1,n]
    \end{cases}.\] In all cases, it is equal to \(n\chi(\overline{a+2}>b)\). This proves \eqn{i+2j}. \[\overline{a+2}-\overline{b-1}+n-2=\begin{cases}
        a+1-b+n=n&\text{if }a+1=b<n\\0&\text{if }[a,b]=[n,1]\text{ or }[n-1,n]
    \end{cases}.\] In either case, it is equal to \(n\chi(\overline{a+2}>\overline{b-1})\). This proves \eqn{i+2j-1}.
\end{proof}

\begin{proof}[Proof of \Theorem{LR}]
Define a \(K^T(\pt)(\!(q_1,q_2)\!)\)-bilinear map \[\mu: QK(X)_q\times QK(X)_q\to QK(X)_q\] so that for \([i,j],\ [k,l]\in\widetilde{W^P}\), \begin{align*}
    \mu(\cO^{[i,j]},\cO^{[k,l]})&=\begin{cases}\cO^{[x,y]}&\text{if } \overline{i}-\overline{j}+\overline{k}-\overline{l}+n-1<n[\chi(\overline{i}>\overline{j})+\chi(\overline{k}>\overline{l})]\\ \cO^{[x,y-1]}+\cO^{[x+1,y]}-\cO^{[x+1,y-1]}&\text{if } \overline{i}-\overline{j}+\overline{k}-\overline{l}+n-1\geq n[\chi(\overline{i}>\overline{j})+\chi(\overline{k}>\overline{l})]\end{cases},    
\end{align*} where $x=i+k-1$, $y=j+l-n$. By \Lemma{alg}, it suffices to show that for all \([i,j],\ [k,l]\in W^P\), \begin{enumerate}
    \item $\mu(\cO^{[1,n]},\cO^{[k,l]})=\cO^{[k,l]}$;
    \item $\mu(\cO^{[i,j]}\star\cO^{[a]},\cO^{[k,l]})=\mu(\cO^{[i,j]},\cO^{[a]}\star\cO^{[k,l]})$ for \(a=1,2\).
\end{enumerate}

(1) follows easily. Direct verification shows that $\mu$ is compatible with $\varphi^*:QK(X)\to QK(X)$. Therefore, for (2), it suffices to show \begin{equation}\label{eqn:assoc}\mu(\cO^{[i,j]}\star\cO^{[1]},\cO^{[k,l]})=\mu(\cO^{[i,j]},\cO^{[1]}\star\cO^{[k,l]}).\end{equation} We now prove \eqn{assoc} case by case.

\item[Case 1: Assume $i+1\not\equiv j$ and $k+1\not\equiv l$.]

\[\mu(\cO^{[i,j]}\star\cO^{[1]},\cO^{[k,l]})
=\mu(\cO^{[i+1,j]},\cO^{[k,l]});\ \mu(\cO^{[i,j]},\cO^{[1]}\star\cO^{[k,l]})=\mu(\cO^{[i,j]},\cO^{[k+1,l]}).\] To show that \[\mu(\cO^{[i+1,j]},\cO^{[k,l]})=\mu(\cO^{[i,j]},\cO^{[k+1,l]}),\] it suffices to show that 
\[(\overline{i+1}-j+k-l+n-1)-n[\chi(\overline{i+1}>j)+\chi(k>l)]\] and \[(i-j+\overline{k+1}-l+n-1)-n[\chi(i>j)+\chi(\overline{k+1}>l)]\] have the same sign. By \eqn{i+1j}, they are both equal to \[i-j+k-l+n-n[\chi(i>j)+\chi(k>l)].\]

\item[Case 2: Assume $i+1\not\equiv j$ and $k+1\equiv l$.]

\[\mu(\cO^{[i,j]}\star\cO^{[1]},\cO^{[k,l]})
=\mu(\cO^{[i+1,j]},\cO^{[k,l]});\]
\[\mu(\cO^{[i,j]},\cO^{[1]}\star\cO^{[k,l]})
=\mu(\cO^{[i,j]},\cO^{[k+1,l-1]}+\cO^{[k+2,l]}-\cO^{[k+2,l-1]}).\]
We have     
\begin{equation*}
    (\overline{i+1}-j+k-l+n-1)-n[\chi(\overline{i+1}>j)+\chi(k>l)]=i-j-n\chi(i>j)+n-1\geq 0,
\end{equation*}
where the equality follows from \eqn{i+1j} and \eqn{=}, and the inequality follows from \eqn{<<}.
Therefore, \[\mu(\cO^{[i+1,j]},\cO^{[k,l]})=\cO^{[i+k,j+l-n-1]}+\cO^{[i+k+1,j+l-n]}-\cO^{[i+k+1,j+l-n-1]}.\] On the other hand, by \eqn{i+1j-1} and \eqn{<<}, we have
\[i-j+\overline{k+1}-\overline{l-1}+n-1=i-j+n\chi(\overline{k+1}>\overline{l-1})<n[\chi(i>j)+\chi(\overline{k+1}>\overline{l-1})];\] by \eqn{i+2j} and \eqn{<<}, we have 
\[i-j+\overline{k+2}-l+n-1=i-j+n\chi(\overline{k+2}>l)<n[\chi(i>j)+\chi(\overline{k+2}>l)];\]
by \eqn{i+2j-1} and \eqn{<<}, we have 
\begin{equation*}
    i-j+\overline{k+2}-\overline{l-1}+n-1=i+1-j+n\chi(\overline{k+2}>\overline{l-1})<n[\chi(i>j)+\chi(\overline{k+2}>\overline{l-1})].
\end{equation*}
Therefore, \begin{align*}&\mu(\cO^{[i,j]},\cO^{[k+1,l-1]}+\cO^{[k+2,l]}-\cO^{[k+2,l-1]})\\=&\cO^{[i+k,j+l-n-1]}+\cO^{[i+k+1,j+l-n]}-\cO^{[i+k+1,j+l-n-1]}\\=&\mu(\cO^{[i+1,j]},\cO^{[k,l]}).\end{align*}

\item[Case 3: Assume $i+1\equiv j$ and $k+1\not\equiv l$.]
This follows from Case 2 by commutativity.

\item[Case 4: Assume $i+1\equiv j$ and $k+1\equiv l$.]
\[\mu(\cO^{[i,j]}\star\cO^{[1]},\cO^{[k,l]})\\
=\mu(\cO^{[i+1,j-1]}+\cO^{[i+2,j]}-\cO^{[i+2,j-1]},\cO^{[k,l]});\]
\[\mu(\cO^{[i,j]},\cO^{[1]}\star\cO^{[k,l]})\\
=\mu(\cO^{[i,j]},\cO^{[k+1,l-1]}+\cO^{[k+2,l]}-\cO^{[k+2,l-1]}).\]
We will prove that 
\begin{equation}\label{eqn:term1}
    \mu(\cO^{[i+1,j-1]},\cO^{[k,l]})=\mu(\cO^{[i,j]},\cO^{[k+1,l-1]}),
\end{equation}
\begin{equation}\label{eqn:term2}
    \mu(\cO^{[i+2,j]},\cO^{[k,l]})=\mu(\cO^{[i,j]},\cO^{[k+2,l]}),
\end{equation}
and 
\begin{equation}\label{eqn:term3}
    \mu(\cO^{[i+2,j-1]},\cO^{[k,l]})=\mu(\cO^{[i,j]},\cO^{[k+2,l-1]}).
\end{equation}
By \eqn{i+1j-1} and \eqn{=}, we have 
\[\overline{i+1}-\overline{j-1}+k-l+n-1=n[\chi(\overline{i+1}>\overline{j-1})+\chi(k>l)]-1\] and
\[i-j+\overline{k+1}-\overline{l-1}+n-1=n[\chi(i>j)+\chi(\overline{k+1}>\overline{l-1})]-1.\] Identity \eqn{term1} then follows from the definition of \(\mu\). By \eqn{i+2j} and \eqn{=}, we have 
\[\overline{i+2}-j+k-l+n-1=n[\chi(\overline{i+2}>j)+\chi(k>l)]-1\] and 
\[i-j+\overline{k+2}-l+n-1=n[\chi(i>j)+\chi(\overline{k+2}>l)]-1,\] which imply \eqn{term2}.
By \eqn{i+2j-1},
\begin{equation*}
    \overline{i+2}-\overline{j-1}+k-l+n-1=n[\chi(\overline{i+2}>\overline{j-1})+\chi(k>l)]
\end{equation*}
and
\begin{equation*}
    i+\overline{k+2}-j-\overline{l-1}+n-1=n[\chi(i>j)+\chi(\overline{k+2}>\overline{l-1})].
\end{equation*} Identity \eqn{term3} follows.
\end{proof}

\begin{cor}\label{cor:genus0}
    For \(d\in H_2(X)^+\), a general \(g\in G\), and \(u,\ v,\ w\in W^P\), the variety \(M_d(g.X^u, X^v, X_w)\) has arithmetic genus \(0\) whenever it is non-empty.
\end{cor}

The following was proven independently in \cite{rosset2022quantum}: for $d \in H_2(X)^+$, a general \(g\in G\), and \(u,\ v,\ w\in W^P\), the variety \(M_d(g.X^u, X^v, X_w)\) is rationally connected. 

To prove \Corollary{genus0}, we shall need the following simple lemma. Recall from \Definition{circPsi} that for \([x,y]\in\widetilde{W^P}\) and \(d\geq d([x,y])\), we define \(\Gamma_d(X^{[x,y]})=\Gamma_{d'}(X^{[\overline{x},\overline{y}]})\), where \(d'=(d'_1,d'_2)=d-d([x,y])\geq 0\). 

\begin{lemma}\label{lemma:iff}
    Let \(z\in W^P\), \([x,y]\in\widetilde{W^P}\), and \(d=(d_1,d_2)\in H_2(X)\) be such that \(X^z=\Gamma_d(X^{[x,y]})\). Then for \([a,b]\in W^P\), \[z\leq [a,b]\Longleftrightarrow x\leq a+d_1 n\text{ and }y\geq b-d_2 n.\]
\end{lemma}
\begin{proof}
    We prove the equivalent statement that 
    \begin{equation*}
        z\leq [a,b]\Longleftrightarrow \overline{x}\leq a+d'_1 n\text{ and }\overline{y}\geq b-d'_2 n.
    \end{equation*}

    Note that for \([a,b],\ [a',b']\in W^P\), \[[a',b']\leq [a,b]\text{ if and only if }a'\leq a\text{ and }b'\geq b.\]

    If \(\overline{x}>a+d'_1 n\), then we must have \(\overline{x}>a\) and \(d'_1=0\). Then by \Section{curvenbhd}, \(z\not\leq [a,b]\). Similarly, if \(\overline{y}<b-d'_2 n\), then \(z\not\leq [a,b]\).

    Conversely, assume \(\overline{x}\leq a+d'_1 n\text{ and }\overline{y}\geq b-d'_2 n\). Then one checks using \Section{curvenbhd} that \(z\leq [a,b]\).
\end{proof}

\begin{proof}[Proof of \Corollary{genus0}]
    We will prove the equivalent statement that 
    \begin{equation*}
        I_d(\cO^u,\cO^v,\cO_w)=\chi(\cO_{M_d(g.X^u, X^v, X_w)})=1\text{ whenever }M_d(g.X^u, X^v, X_w)\neq\varnothing,
    \end{equation*}
    where the first equality uses \cite[Section 4.1]{qkgrass} and the references therein.

    We have 
    \begin{equation}\label{eqn:I_I^vee}
        I_d(\cO^u,\cO^v,\cO_w)=\sum_{z\in{W^P}:\ z\leq w}I_d(\cO^u,\cO^v,\cO_z^\vee),
    \end{equation}
    and
    \begin{equation}\label{eqn:I^vee_Psi}
        \sum_{z\in{W^P},\ d\geq0}I_d(\cO^u,\cO^v,\cO_z^\vee)q^d \cO^z=\cO^u\odot\cO^v=\Psi(\cO^u\star\cO^v).
    \end{equation}

    Assume \(M_d(g.X^u, X^v, X_w)\neq\varnothing\). Let \(d=(d_1,d_2)\). Then 
    \begin{equation}\label{eqn:non-emptys}
        M_{d_c}(g.p_c(X^u),p_c(X^v),p_c(X_w))\neq\varnothing\text{ for }c=1,2.
    \end{equation} Let \(u=[i,j],\ v=[k,l],\ w=[a,b]\). By dimension count, \eqn{non-emptys} implies that
    \begin{equation}\label{eqn:xy}
        x\coloneqq i+k-1\leq a+d_1 n\text{ and }y\coloneqq j+l-n\geq b-d_2 n.
    \end{equation}
    \item[Case 1: $x-y<n{[}\chi(i>j)+\chi(k>l){]}$.]\label{itm}
    By \Theorem{LR},
    \begin{equation}\label{eqn:Psi1}
        \Psi(\cO^u\star\cO^v)=\Psi(\cO^{[x,y]}).
    \end{equation}
    Let \(X^{z_0}=\Gamma_d(X^{[x,y]})\). By \eqn{I^vee_Psi}, \eqn{Psi1}, and \eqn{Psi}, 
    \begin{equation}\label{eqn:1Id}
        I_d(\cO^u,\cO^v,\cO_z^\vee)=
        \delta_{zz_0}.
    \end{equation}
    By \Lemma{iff}, condition \eqn{xy} is equivalent to \(z_0\leq w\). By \eqn{I_I^vee} and \eqn{1Id}, \[I_d(\cO^u,\cO^v,\cO_w)=1.\]
    \item[Case 2: $x-y\geq n{[}\chi(i>j)+\chi(k>l){]}$.] By \Theorem{LR},
    \begin{equation}\label{eqn:Psi2}
        \Psi(\cO^u\star\cO^v)=\Psi(\cO^{[x,y-1]}+\cO^{[x+1,y]}-\cO^{[x+1,y-1]}).
    \end{equation}
    
    Let \(\Gamma_d(X^{[x,y-1]})=X^{z_1}\), \(\Gamma_d(X^{[x+1,y]})=X^{z_2}\), and \(\Gamma_d(X^{[x+1,y-1]})=X^{z_3}\). By \eqn{I^vee_Psi}, \eqn{Psi2}, and \eqn{Psi}, 
    \begin{equation}\label{eqn:2Id}
        I_d(\cO^u,\cO^v,\cO_z^\vee)=
        \delta_{zz_1}+\delta_{zz_2}-\delta_{zz_3}.
    \end{equation}
    By \Lemma{iff}, \eqn{I_I^vee}, and \eqn{2Id} we have: 
    \begin{enumerate}
        \item if \(x+1\leq a+d_1 n\) and \(y-1\geq b-d_2 n\), then \(z_1,\ z_2,\ z_3\leq w\), and  \[I_d(\cO^u,\cO^v,\cO_w)=1+1-1=1;\]
        \item if \(x=a+d_1 n\) and \(y-1\geq b-d_2 n\), then \(z_1\leq w\) and \(z_2,\ z_3\not\leq w\), and \[I_d(\cO^u,\cO^v,\cO_w)=1;\]
        \item if \(x+1\leq a+d_1 n\) and \(y=b-d_2 n\), then \(z_2\leq w\) and \(z_1,\ z_3\not\leq w\), and \[I_d(\cO^u,\cO^v,\cO_w)=1.\]
    \end{enumerate}
    In view of \eqn{xy}, the only remaining possibility is 
    \begin{equation}\label{eqn:==}
        x=a+d_1 n,\ y=b-d_2n.
    \end{equation}
    Assuming \eqn{==}, \(M_d(g.X^u,X^v,X_w)\neq\varnothing\) implies by dimension count that \[d_1+d_2+\chi(a>b)\leq \chi(i>j)+\chi(k>l)\leq \frac{x-y}{n}.\] Therefore, \(n\chi(a>b)\leq a-b\), which is a contradiction. 
\end{proof}

In the non-equivariant case, positivity is verified by \Theorem{LR} as follows. First, recall from \Definition{length} that for \([x,y]\in \widetilde{W^P}\), we let 
\[\ell([x,y])=x-y-1+n-\chi(\overline{x}>\overline{y})-d_1-d_2,\text{ where }(d_1,d_2)=d([x,y]).\]

\begin{cor}\label{cor:LRpos}
For \(u,\ v\in{W^P}\), \(w\in\widetilde{W^P}\), let \(N_{u,v}^{w}\) be the coefficient of \(\cO^w\) in the product \(\cO^u\star\cO^v\) in \(QK(X)\), then \[(-1)^{\ell(u)+\ell(v)+\ell(w)}N_{u,v}^{w}\in\N.\]
\end{cor}
\begin{proof}
    Let \[u=[i,j],\ v=[k,l],\ x=i+k-1,\ y=j+l-n.\] By \Theorem{LR}, it suffices to show that \begin{enumerate}
        \item when \(x-y<n[\chi(i>j)+\chi(k>l)]\), \[\ell(u)+\ell(v)+\ell([x,y])\text{ is even};\]
        \item when \(x-y\geq n[\chi(i>j)+\chi(k>l)]\), \[\ell(u)+\ell(v)+\ell([x,y-1])\text{ and }\ell(u)+\ell(v)+\ell([x+1,y])\text{ are even};\]\[\ell(u)+\ell(v)+\ell([x+1,y-1])\text{ is odd}.\]
    \end{enumerate}
    
    Note that \[\ell(u)+\ell(v)=x-y-1+n-[\chi(i>j)+\chi(k>l)].\] 

    First assume \begin{equation*}x-y<n[\chi(i>j)+\chi(k>l)].\end{equation*} It suffices to show that \begin{equation}\label{eqn:even}
        \chi(i>j)+\chi(k>l)+\chi(\overline{x}>\overline{y})+d_1+d_2\text{ is even}.
    \end{equation}
     When \(i>j\) and \(k>l\), we have \[n<x-y=\overline{x}-\overline{y}+n(d_1+d_2)<2n\] and therefore, either \[\overline{x}>\overline{y}\text{ and } d_1+d_2=1\] or \[\overline{x}<\overline{y}\text{ and }d_1+d_2=2.\] When \(i>j\) and \(k<l\), or, \(i<j\) and \(k>l\), we have \[0<x-y=\overline{x}-\overline{y}+n(d_1+d_2)<n\] and therefore, either \[\overline{x}>\overline{y}\text{ and } d_1+d_2=0\] or \[\overline{x}<\overline{y}\text{ and }d_1+d_2=1.\] When \(i<j\) and \(k<l\), we have \[-n<x-y=\overline{x}-\overline{y}+n(d_1+d_2)<0\] and therefore \[\overline{x}<\overline{y}\text{ and }d_1+d_2=0.\] In all cases, \eqn{even} holds.

    Now assume \begin{equation*}x-y\geq n[\chi(i>j)+\chi(k>l)].\end{equation*} One checks that \[\ell([x,y-1])=\ell([x+1,y])=\ell([x+1,y-1])-1.\] Therefore, it suffices to show that \[\ell(u)+\ell(v)+\ell([x,y-1])\text{ is even}.\] The proof of this is completely analogous and is therefore omitted.
\end{proof}

For the quantum cohomology and quantum \(K\)-theory of all homogeneous spaces \(G/P\), it is known that there is a unique minimal power of the quantum parameters \(q\) that appear with nonzero coefficient in the product of any two given Schubert classes. This minimal power is the same in quantum cohomology and quantum \(K\)-theory, and corresponds to the minimal degree of a stable curve connecting opposite Schubert varieties \cite{Euler}. For the quantum cohomology of cominuscule flag varieties, it is known that the powers of \(q\) that appear with nonzero coefficients in the product of any two given Schubert classes form an interval between two degrees \cites{qkpos, Postnikov}. However, this is not true for the quantum cohomology of complete flag varieties. A counterexample is the product of \([Y^{164532}]\) with itself in the quantum cohomology ring of \(Y=\Fl(6)\), where the maximal powers of \(q\) with nonzero coefficients are not unique and the powers of \(q\) with nonzero coefficients do not form an interval \cite{qkpos}. For the quantum \(K\)-theory of type \(A\) complete flag varieties, it is known that there is a unique maximal power of \(q\) that appear with nonzero coefficient in the product of a Schubert divisor class and a Schubert class; however, the powers of \(q\) that appear with nonzero coefficients don't always form an interval \cite{lenart}. It remains an open problem  whether there is always a unique maximal power of \(q\) that appear with nonzero coefficient in the product of two given Schubert classes in the quantum \(K\)-theory ring of a homogeneous space. In the case of incidence varieties, it is immediate from \Theorem{LR} that:

\begin{cor}\label{cor:power}
    In \(QK(X)\), for \(u,\ v\in W^P\), the powers of \(q\) that appear with nonzero coefficients in \(\cO^u\star\cO^v\) form an interval between two degrees.
\end{cor}

Using \cite[Theorem 2.18]{kato} with the projection \(p_1: X\to\P^{n-1}\), we get the following closed formula for quantum \(K\)-theoretic Littlewood--Richardson coefficients for projective spaces, stated in terms of \Notation{proj}. This is a special case of \cite[Theorem 5.4]{qkgrass}.

\begin{cor}\label{cor:LRP^n}
    In \(QK(\P^{n-1})\), for \(a,\ b\in\{0,\dots,n-1\}\),
    \[\cO^a\star\cO^b=\cO^{a+b}.\]
\end{cor}

In particular, the quantum \(K\)-theory ring and quantum cohomology rings are isomorphic for projective spaces.

\bibliography{bib.bib}
\end{document}